\documentclass[a4paper,reqno,11pt]{amsart}

\usepackage[margin=1.2in]{geometry}
\usepackage[onehalfspacing]{setspace}
\usepackage[english]{babel}
\usepackage{amsfonts}
\usepackage{amssymb}
\usepackage[reqno]{amsmath}
\usepackage{amsthm}
\usepackage{graphicx}
\usepackage{latexsym}
\usepackage{mathtools}
\usepackage{ bbold }
\usepackage[justification=centering]{caption}
\usepackage{tikz}
\usepackage{tikz-cd}
\usetikzlibrary{patterns,matrix,through,arrows,decorations.pathreplacing,decorations.markings,shadows,shapes.geometric,positioning,calc,backgrounds,fit}
\usetikzlibrary{arrows}
\usetikzlibrary{knots, intersections, decorations.pathreplacing, decorations.pathmorphing, hobby, backgrounds}
\usetikzlibrary{external}
\tikzsetexternalprefix{cat-figures/}
\tikzset{commutative diagrams/.cd, arrow style=tikz, diagrams={>=Computer Modern Rightarrow}, path details/.style={every node/.style={midway, sloped, font=\footnotesize}}}
\tikzset{anchorbase/.style={baseline={([yshift=-0.5ex]current bounding box.center)}}, int/.style={thick}, cross line/.style={preaction={draw=white,line width=6pt,-}}, wall/.style={thin,double,blue}, middlearrow/.style={postaction=decorate,decoration={markings,mark=at position .55 with {\arrow{stealth};}}}, middlearrowrev/.style={postaction=decorate,decoration={markings,mark=at position .55 with {\arrowreversed{stealth};}}}, ev/.style={shape=rectangle, draw}}
\usepackage{t-angles}
\usepackage{setspace}
\usetikzlibrary{cd}
\usepackage[all,cmtip]{xy}
\usepackage[hidelinks]{hyperref}
\usepackage{enumitem}
\usepackage[capitalise,nameinlink,noabbrev]{cleveref}
\usepackage[textsize=footnotesize]{todonotes}
\usepackage[final]{microtype}

\theoremstyle{plain}
\newtheorem{satz}{Satz}[section]
\newtheorem{theorem}[satz]{Theorem}

\newtheorem{lemma}[satz]{Lemma}
\newtheorem{proposition}[satz]{Proposition}
\newtheorem{corollary}[satz]{Corollary}
\newtheorem{introtheorem}{Theorem}

\theoremstyle{definition}
\newtheorem{definitionx}[satz]{Definition}
\newenvironment{definition}
{\pushQED{\qed}\definitionx}
{\popQED\enddefinitionx}

\newtheorem{examplex}[satz]{Example}
\newenvironment{example}
{\pushQED{\qed}\examplex}
{\popQED\endexamplex}
\newtheorem{remarkx}[satz]{Remark}
\newenvironment{remark}
{\pushQED{\qed}\remarkx}
{\popQED\endremarkx}

\newcommand{\sltwo}{\mathfrak{sl}_2(\mathbb{C})}

\newcommand{\sunrolled}{\overline{U}^H_q(\mathfrak{sl}_2(\mathbb{C}))}
\newcommand{\Bunrolled}{\overline{U}^{H}_q(\mathfrak{b})}

\newcommand{\id}{\mathrm{id}}

\newcommand{\RI}{\text{RI}}
\newcommand{\Hom}{\textnormal{Hom}}
\newcommand{\End}{\textnormal{End}}

\newcommand{\ev}{\mathrm{ev}}
\newcommand{\coev}{\mathrm{coev}}
\newcommand{\C}{\mathbb{C}}
\newcommand{\Z}{\mathbb{Z}}
\newcommand{\Q}{\mathbb{Q}}
\newcommand{\RII}{\text{RII}}
\newcommand{\RIII}{\text{RIII}}
\newcommand{\Span}{\text{span}}
\newcommand{\ch}{\text{ch}}
\newcommand{\qbinom}{\genfrac{[}{]}{0pt}{}}
\newcommand{\qdim}{\textnormal{qdim}}
\newcommand{\ptr}{\mathrm{ptr}}
\newcommand{\lmod}{{\operatorname{-mod}}}
\newcommand{\writhe}{{\operatorname{wr}}}
\newcommand{\Rib}{\mathbf{Rib}}
\newcommand*{\spr}[2]{\left(\,#1\,\middle|\,#2\,\right)}

\title[Renormalized Reshetikhin--Turaev invariants]{Renormalized Reshetikhin--Turaev invariants for the unrolled quantum group of $\sltwo$}

\author[N. Geer]{Nathan Geer}
\address{Department of Mathematics and Statistics \\ Utah State University\\
	Logan, Utah 84322 \\ USA}
\email{nathan.geer@gmail.com}

\author[A. Robertson]{Adam Robertson}
\email{adam.robertson@usu.edu}

\author[J.-L. Spellmann]{Jan-Luca Spellmann}
\email{jspellmann@outlook.de}

\author[M. B. Young]{Matthew B. Young}
\email{matthew.young@usu.edu}

\begin{document}

\date{\today}
\keywords{Reshetikhin--Turaev theory. Representation theory of quantum groups.}
\subjclass[2010]{Primary: 81R50; Secondary 57M25.}

\begin{abstract}
    This paper is a self-contained introduction to the theory of renormalized Reshetikhin--Turaev invariants of links defined by Geer, Patureau-Mirand and Turaev. Whereas the standard Reshetikhin--Turaev theory of a $\C$-linear ribbon category assigns the trivial invariant to any link with a component colored by a simple object of vanishing quantum dimension, the renormalized theory does not. We give a streamlined development of the renormalized Reshetikhin--Turaev theory of links for the category of weight modules over the restricted unrolled quantum group of $\sltwo$ at an even root of unity.
\end{abstract}
	
\maketitle
\setcounter{tocdepth}{1}
\tableofcontents
    
\section*{Introduction}\label{intro}
\subsection*{Background and motivation}

Let $L \subset S^3$ be a framed link and $M_L$ the closed orientable $3$-manifold obtained from $S^3$ by surgery along $L$. By a theorem of Lickorish and Wallace, any closed connected orientable 3-manifold arises in this way \cite{wallace1960,lickorish1962}. Moreover, the $3$-manifolds $M_L$ and $M_{L^{\prime}}$ are homeomorphic if and only if the framed links $L$ and $L^{\prime}$ are related by a finite sequence of Kirby moves \cite{kirby1978calculus}. These results are the starting point for a knot theoretic approach to problems and constructions in $3$-manifold topology. For example, it follows from the previous two results that an isotopy invariant of framed links which is also invariant under Kirby moves defines an invariant of $3$-manifolds, thereby emphasizing the topological importance of link invariants.

Reshetikhin and Turaev constructed a large class of link invariants using the theory of ribbon categories \cite{reshetikhin1990ribbon}. Associated to each ribbon category $\mathcal{D}$ is a ribbon functor $F_{\mathcal{D}} : \Rib_{\mathcal{D}} \rightarrow \mathcal{D}$ with domain the category of $\mathcal{D}$-colored ribbon graphs. Interpreting an isotopy class of a $\mathcal{D}$-colored framed link $L$ as a $(0,0)$-tangle, and so an endomorphism of the unit object $\mathbb{I} \in \Rib_{\mathcal{D}}$, produces an invariant $F_{\mathcal{D}}(L) \in \End_{\mathcal{D}}(\mathbb{I})$ of $L$, the \emph{Reshetikhin--Turaev invariant}. The invariant $F_{\mathcal{D}}(L)$ is computed as follows. Choose a regular diagram $D$ for $L$. Decompose $D$ into elementary pieces consisting of cups, caps, simple crossings and twists and assign to these pieces the corresponding coevaluations, evaluations, braidings and twists, respectively, of $\mathcal{D}$. The composition of these morphisms in $\mathcal{D}$ is $F_{\mathcal{D}}(L)$. \cref{fig:RTHopf} illustrates this procedure for the Hopf link.

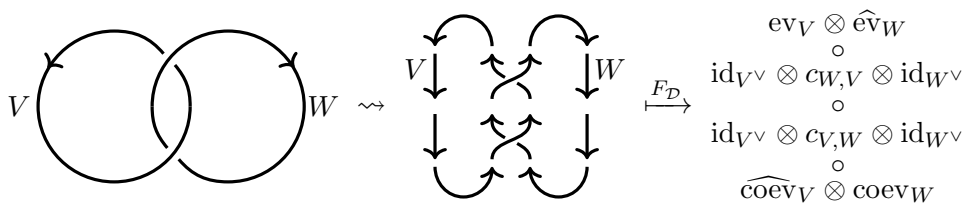
\begin{figure}
\centering
\begin{equation*}
    \begin{aligned}
    \begin{tikzpicture}[anchorbase]
        \draw [very thick] (-.25,0) to [out=270,in=130] (-.05,-.55);
        \draw [very thick, postaction={decorate}, decoration={markings, mark=at position .55 with {\arrow{<}}}] (.1,-.7) to [out=310,in=180] (.75,-1) to [out=0,in=270] (1.75,0) to [out=90,in=0] (.75,1) to [out=180,in=90] (-.25,0);
        \draw [very thick, , postaction={decorate}, decoration={markings, mark=at position .7 with {\arrow{<}}}] (.25,0) to [out=270,in=0] (-.75,-1) to [out=180,in=270] (-1.75,0) to [out=90,in=180] (-.75,1) to [out=0,in=130](-.1,.7);
        \draw [very thick ] (.05,.55) to [out=310,in=90] (.25,0);
        \node at (-2,0) {$V$};
        \node at (2,0) {$W$};
    \end{tikzpicture}
    \rightsquigarrow 
    \begin{tikzpicture}[anchorbase]
        \draw [->,very thick] (-.25,.1) to [out=90,in=270] (.25,.6) to (.25,.7);
        \draw [very thick] (.25,.1) to [out=90,in=330] (.08,.31);
        \draw [->,very thick] (-.08,.385) to [out=150,in=270] (-.25,.6) to (-.25,.7);
        \draw [<-,very thick] (.25,-.1) to (.25,-.2) to [out=270,in=90] (-.25,-.7);
        \draw [very thick] (.25,-.7) to [out=90,in=330] (.08,-.49);
        \draw [->,very thick] (-.08,-.415) to [out=150,in=270] (-.25,-.2) to (-.25,-.1);
        \draw [<-,very thick] (-1,.8) to [out=90,in=180] (-.625,1.2) to [out=0,in=90] (-.25,.8);
        \draw [<-,very thick] (1,.8) to [out=90,in=0] (.625,1.2) to [out=180,in=90] (.25,.8);
        \draw [->,very thick] (-1,-.8) to [out=270,in=180] (-.625,-1.2) to [out=0,in=270] (-.25,-.8);
        \draw [->,very thick] (1,-.8) to [out=270,in=0] (.625,-1.2) to [out=180,in=270] (.25,-.8);
        \draw [->,very thick] (-1,.7) to (-1,.1);
        \draw [<-,very thick] (-1,-.7) to (-1,-.1);
        \draw [->,very thick] (1,.7) to (1,.1);
        \draw [<-,very thick] (1,-.7) to (1,-.1);
        \node at (-1.25,.5) {$V$};
        \node at (1.3,.5) {$W$};
    \end{tikzpicture}
    \xmapsto{F_\mathcal{D}}
    \begin{tikzpicture} [anchorbase]
        \node at (0,1.1) {$\ev_V \otimes \widehat{\ev}_W$};
        \node at (0,.4) {$\id_{V^\vee} \otimes c_{W,V} \otimes \id_{W^\vee}$};
        \node at (0,-.4) {$\id_{V^\vee} \otimes c_{V,W} \otimes \id_{W^\vee}$};
        \node at (0,-1.1) {$\widehat{\coev}_V \otimes \coev_W$};
        \node at (0,-.8) {$\circ$};
        \node at (0,0) {$\circ$};
        \node at (0,.73) {$\circ$};
    \end{tikzpicture}
    \end{aligned}
\end{equation*}
\caption{The Reshetikhin--Turaev invariant of a Hopf link colored by objects $V$ and $W$ of $\mathcal{D}$. Composition in $\mathcal{D}$ is read from bottom to top.}
\label{fig:RTHopf}
\end{figure}
	
The Reshetikhin--Turaev construction highlights the topological significance of ribbon categories. Classical representation theory produces many examples of symmetric monoidal categories: representations of groups and Lie algebras and, more generally, cocommutative Hopf algebras. Unfortunately, Reshetikhin--Turaev invariants associated to a symmetric monoidal category are uninteresting since they retain information only about the number of components of a link. On the other hand, categories of representations of quantum groups and, more generally, quasi-triangular Hopf algebras famously give rise to (non-symmetric) ribbon categories \cite{jimbo1985aq,drinfeld1986quantum,drinfeld1990hopf,chari1994}. The resulting quantum invariants of links, which include the Jones and HOMFLYPT polynomials, are at the foundation of quantum topology \cite{jones1987,freyd1985,przytycki1988,reshetikhin1990ribbon,turaev2016quantum}.

Many ribbon categories arising in representation theory have the following properties:
\begin{enumerate}[label=(P\arabic*)]
    \item \label{ite:vanQD} The category has simple objects with vanishing quantum dimension.
	\item \label{ite:nonss} The category is non-semisimple, that is, not every short exact sequence splits.
	\item \label{ite:infSimp} The category has infinitely many non-isomorphic simple objects.
\end{enumerate}
For example, the category $U_q(\mathfrak{g}) \lmod$ of finite dimensional representations of the quantum group $U_q(\mathfrak{g})$ associated to a complex simple Lie (super)algebra $\mathfrak{g}$ at a root of unity has Properties \ref{ite:vanQD}-\ref{ite:infSimp}. It is well-known that the Reshetikhin--Turaev invariant of a link colored by a simple object of vanishing quantum dimension is zero. For this reason, the Reshetikhin--Turaev construction is not well-suited to extracting the full topological content of categories having Property \ref{ite:vanQD}. Properties \ref{ite:nonss} and \ref{ite:infSimp} do not cause problems for Reshetikhin--Turaev invariants of links but are serious obstructions to extending these invariants to $3$-manifolds. For example, these properties obstruct the definition of the Kirby color, a weighted sum of isomorphism classes of simple objects, which is crucial to the construction of $3$-manifold invariants in \cite{reshetikhin1991invariants}.

A standard approach to simultaneously eliminating Properties \ref{ite:vanQD}-\ref{ite:infSimp} for the category $U_q(\mathfrak{g}) \lmod$, with $\mathfrak{g}$ a simple Lie algebra, is semisimplification \cite{andersen1992}, whereby simple objects of vanishing quantum dimension are formally set to zero. The semisimpflied categories are, for particular roots of unity, modular tensor categories. The resulting $3$-manifold invariants comprise the top level of a three dimensional topological quantum field theory which is a mathematical model for Chern--Simons theory with gauge group the simply connected compact Lie group associated to $\mathfrak{g}$ \cite{witten1989,reshetikhin1991invariants}. On the other hand, for the category $ U_q(\mathfrak{g}) \lmod$, with $\mathfrak{g}$ a type I Lie superalgebra, typical representations have vanishing quantum dimension and semisimplification eliminates most interesting content of the category.
	
Ribbon categories with Properties \ref{ite:vanQD}-\ref{ite:infSimp} also arise in quantum field theory. For example, such categories arise as line operators in Chern--Simons theories with non-compact gauge groups \cite{witten1991,barnatan1991,rozansky1994,mikhaylov2015} and topological twists of supersymmetric quantum field theories \cite{kapustin2009b,creutzig2021} and as modules for vertex operator algebras in non-rational (or logarithmic) conformal field theories \cite{rozansky1993,creutzig2013,creutzig2013b}.

Early examples of knot invariants constructed from ribbon categories with Properties \ref{ite:vanQD}-\ref{ite:infSimp} include the work of Akutsu, Deguchi and Ohtsuki \cite{akutsu1992} and Murakami and Murakami \cite{murakami2001}, who defined (framed) link invariants from typical representations of the unrolled quantum group $\sunrolled$ at an even root of unity. A systematic program to define and study quantum invariants from ribbon categories with Properties \ref{ite:vanQD}-\ref{ite:infSimp}, called \emph{renormalized Reshetikhin--Turaev theory}, was developed by Blanchet, Costantino, Geer, Patureau-Mirand and Turaev \cite{BCGP,geer_2009,costantino2014quantum}. In the setting of links, these renormalized invariants provide non-trivial invariants of links colored by objects with vanishing quantum dimension. The goal of this paper is to give a self-contained introduction to the theory of renormalized Reshetikhin--Turaev invariants of links in the simplest case of the category of modules over $\sunrolled$, following \cite{geer_2009,costantino2015}. While this paper contains no new results, we do offer a number of new proofs of known results and give complete details where they are often not available in the literature. Some familiarity with the representation theory of $U_q(\sltwo)$, at the level of \cite{jantzen}, and its associated Reshetikhin--Turaev invariants would be beneficial, but is not strictly necessary. We assume basic knowledge of Hopf algebras and monoidal categories.

\subsection*{Contents of this paper}

Fix an integer $r \geq 2$ and set $q= e^{\frac{\pi \sqrt{-1}}{r}}$. The De Concini--Kac quantum group $U_q(\sltwo)$ has generators $K^{\pm 1}$, $E$ and $F$ with relations $KK^{-1}=1=K^{-1}K$ and
\begin{equation}
\label{eq:DCKReln}
KE = q^2EK, \qquad KF = q^{-2}FK, \qquad EF - FE = \frac{K - K^{-1}}{q - q^{-1}}.
\end{equation}
The \emph{unrolled quantum group} $U_q^H(\sltwo)$, as introduced in \cite{geer_2009}, is defined similarly to the De Concini--Kac quantum group but with an additional generator $H$, thought of as a logarithm of $K$, which commutes with $K$ and satisfies the classical limit of the first two relations \eqref{eq:DCKReln}:
\[
[H,E] = 2E, \qquad [H,F]=-2F.
\]
The algebra of primary interest in this paper is the \emph{restricted unrolled quantum group} $\sunrolled$, defined to be the quotient of $U^H_q(\sltwo)$ by the relations $E^r = F^r = 0$. A $\sunrolled$-module $V$ is called a \emph{weight module} if it is a direct sum of $H$-eigenspaces and $K=q^H$ as operators on $V$. The category $\mathcal{C}$ of finite dimensional weight modules over $\sunrolled$ is the central algebraic object of this paper.
	
\cref{sunrolled} is devoted to a detailed study of $\mathcal{C}$. A natural Hopf algebra structure on $\sunrolled$ gives $\mathcal{C}$ the structure of a rigid monoidal abelian category. We use Verma modules, which are finite dimensional due to the relations $E^r=F^r=0$, to classify simple objects of $\mathcal{C}$ in \cref{simplemodules}. The result is that there is a discrete family of simple modules $S^{lr}_n$ of highest weight $n + lr$ and dimension $n+1$, $l \in \Z$, $0 \leq n \leq r-2$, and a continuous family of simple Verma modules $V_{\alpha}$ of highest weight $\alpha + r -1$, $\alpha \in \C \setminus \Z \cup r \Z$, and dimension $r$.
    
Tracking $H$-weights modulo $2 \Z$ defines a $\C \slash 2 \Z$-grading $\mathcal{C} = \bigoplus_{\overline{\alpha} \in \C \slash 2 \Z} \mathcal{C}_{\overline{\alpha}}$ which is compatible with the rigid monoidal structure. While the category $\mathcal{C}$ is not semisimple, it is \emph{generically semisimple} in the sense that most homogeneous subcategories $\mathcal{C}_{\overline{\alpha}} \subset \mathcal{C}$ are semisimple. More precisely, we prove in \cref{Cisgss} that $\mathcal{C}_{\overline{\alpha}}$ is semisimple unless $\overline{\alpha} \in \Z \slash 2 \Z$. \cref{ribbon} states that $\mathcal{C}$ is braided. A complete proof of this statement does not seem to be in the literature. The proof we present is elementary and self-contained. The form of the braiding is motivated by the well-known universal $R$-matrix of the $\hbar$-adic quantum group of $\sltwo$ \cite{drinfeld1986quantum,ohtsuki2002quantum}. In \cref{ribbon theorem} we prove that $\mathcal{C}$ is ribbon. The candidate ribbon structure is based on the twist associated to the rigid monoidal structure, namely the right partial trace of the braiding. We use generic semisimplicity of $\mathcal{C}$ to prove that this twist is compatible with duality by checking that this is so generically and concluding, via a general result of \cite{geer_2017}, that this extends to the entirety of $\mathcal{C}$. The results of \cref{sunrolled} can be summarized as follows.

\begin{introtheorem}
    The category $\mathcal{C}$ is a $\mathbb{C}/2\Z$-graded generically semisimple ribbon category.
\end{introtheorem}
	
In \cref{rti} we recall standard material related to the Reshetikhin--Turaev functor $F_{\mathcal{D}}: \mathbf{Rib}_\mathcal{D} \rightarrow \mathcal{D}$ associated to a ribbon category $\mathcal{D}$. Central to the renormalized theory is the well-known statement, proved in this paper as \cref{cutting}, that if $V \in \mathcal{D}$ is a simple object of a $\mathbb{C}$-linear ribbon category, $L$ is a $\mathcal{D}$-colored link and $T$ is a $(1,1)$-tangle with closure $L$ and open strand colored by $V$, then
\begin{equation*}
	F_\mathcal{D}(L) = \qdim_{\mathcal{D}}(V) F_\mathcal{D}(T).
\end{equation*}
Here both sides of the equation are identified with the scalar by which they act. In particular, if $\qdim_{\mathcal{D}}(V)=0$, then $F_\mathcal{D}(L)$ vanishes, while $F_\mathcal{D}(T)$ need not. We prove in \cref{modified_invariant} that it, if $L$ is a knot, then $F_\mathcal{D}(T)$ is an invariant of $L$.
	
In \cref{mqi} we extend the invariant $L \mapsto F_{\mathcal{D}}(T)$ from framed knots to framed links. To clarify the exposition, we restrict attention to $\mathcal{D}=\mathcal{C}$, the category of weight $\sunrolled$-modules. The main obstacle in this extension is that cutting a link $L$ with multiple components produces a $(1,1$)-tangle whose isotopy type depends on the component which is cut. Ambidextrous modules are the key to overcoming this obstacle. A simple $\sunrolled$-module $V$ is called \emph{ambidextrous} if the equality
\begin{equation*}
	\begin{aligned}
		F_\mathcal{C}\left(
	    \begin{tikzpicture}[anchorbase,scale=0.8]
			\draw [very thick] (-.5,-.25) -- (.5,-.25) -- (.5,.25) -- (-.5,.25) -- cycle;
			\draw [<-, very thick] (-1.25,0) to [out=90, in=180] (-.75,1) to [out=0,in = 90] (-.25,.25);
			\draw [very thick] (-1.25,0) to [out=270, in=180] (-.75,-1) to [out=0,in=270] (-.25,-.25);
			\draw [->, very thick] (.25,.25) to [out=90, in=230] (.5,1);
			\draw [very thick] (.25,-.25) to [out=270, in=130] (.5,-1);
			\node at (-1.7,0) {$V$};
			\node at (.8,-.9) {$V$};
			\node at (0,0) {$T$};
		\end{tikzpicture}
		\right)
		&=F_\mathcal{C}\left(
		\begin{tikzpicture}[anchorbase,scale=0.8]
			\draw [very thick] (-.5,-.25) -- (.5,-.25) -- (.5,.25) -- (-.5,.25) -- cycle;
			\draw [->, very thick] (-.25,.25) to [out=90, in=310] (-.5,1);
			\draw [very thick] (-.25,-.25) to [out=270, in=50] (-.5,-1);
			\draw [<-, very thick] (1.25,0) to [out=90, in=0] (.75,1) to [out=180,in = 90] (.25,.25);
			\draw [very thick] (1.25,0) to [out=270, in=00] (.75,-1) to [out=180,in=270] (.25,-.25);
			\node at (1.7,0) {$V$};
			\node at (-.8,-.9) {$V$};
			\node at (0,0) {$T$};
		\end{tikzpicture}
		\right)
	\end{aligned}
\end{equation*}
of endomorphisms of $V$ holds for all (2,2)-tangles $T$ whose open strands are colored by $V$. We prove in \cref{manyAmbi} that all simple $\sunrolled$-modules are ambidextrous.
	
Define a function $S': \mathbb{C} \times \mathbb{C} \rightarrow \mathbb{C}$ by
\begin{equation*}
	\begin{aligned}
	    S'(\beta, \alpha) = \,	F_\mathcal{C}\left(
	    \begin{tikzpicture}[anchorbase,scale=0.8]
			\draw [very thick, postaction={decorate}, decoration={markings, mark=at position 0.4 with {\arrow{>}}}] (-.1,-.5) to [out=180,in=270] (-.8,0) to [out=90,in=180] (0,.5) to [out=0,in=90] (.8,0) to [out=270,in=0] (.1,-.5);
			\draw [very thick] (0,-1) to (0,.4);
			\draw [->, very thick] (0,.6) to (0,1);
			\node at (1.2,0) {$V_\beta$};
			\node at (.45,-.9) {$V_\alpha$};
		\end{tikzpicture}
		\right) \in \End_{\mathcal{C}}(V_{\alpha}) \simeq \mathbb{C}.
	\end{aligned}
\end{equation*}
For a fixed ambidextrous module $V_\eta$, the \emph{modified quantum dimension} of a simple Verma module $V_{\alpha}$ is defined to be $\mathbf{d}_\eta(\alpha) = \frac{S'(\alpha, \eta)}{S'(\eta, \alpha)}$. The main result of this paper can be stated as follows.
	
\begin{introtheorem}\textup{(\cref{invaraiant})}
    Let $V_\eta \in \mathcal{C}$ be an ambidextrous module and $L$ a framed link with at least one strand colored by $V_\alpha$ for some $\alpha \in \C \setminus \Z \cup r \Z$. Then the assignment
		\begin{equation*}
			L \mapsto F_\eta'(L) := \mathbf{d}_\eta(\alpha) F_\mathcal{C}(T),
		\end{equation*}
	where $T$ is a $(1,1)$-tangle whose closure is $L$ and whose open strand is colored by $V_{\alpha}$, is a well-defined isotopy invariant of framed colored links.
\end{introtheorem}

In \cref{sec:examples} we discuss some basic properties of the renormalized invariant $F'_\eta$, such as its behavior under connect sum and its associated skein relations, and compute some basic examples. We also show that renormalizations with respect to different ambidextrous modules $V_{\eta}$ lead to invariants which differ by a global scalar.
	
Finally, in \cref{sec:furtherReading} we present a brief guide to further mathematical and physical applications of renormalized Reshetikhin--Turaev invariants of links.
    
\subsection*{Conventions} 
The ground field is $\mathbb{C}$. Write $\otimes$ for $\otimes_{\C}$. All modules are left modules and finite dimensional over $\C$. Any categorical notion regarding monoidal categories is in congruence with \cite{etingof2016tensor}. Given a scalar endomorphism $e$ of a vector space $V$, define $\langle e \rangle \in \C$ by $e = \langle e \rangle \cdot \id_V$.
		
\subsection*{Acknowledgements}
N.G.\ is partially supported by NSF grants DMS-1664387 and DMS-2104497. M.B.Y. is partially supported by a Simons Foundation Collaboration Grant for Mathematicians (Award ID 853541).

\addtocontents{toc}{\protect\setcounter{tocdepth}{2}}

\section{The unrolled quantum group $\sunrolled$ and its weight modules} \label{sunrolled}

Fix an integer $r \geq 2$. Set $q = e^{\frac{\pi \sqrt{-1}}{r}}$.
For $z \in \mathbb{C}$, define
\[
q^z = e^{\frac{\pi \sqrt{-1} z}{r}}, \qquad \{z\} = q^z - q^{-z}, \qquad [z] = \frac{\{z\}}{\{1\}}.
\]
Set $\{0\}!=1$ and $\{n\}! = \prod_{i=1}^n \{i\}$ for $n \in \Z_{> 0}$, and similarly for $[n]!$. For $0 \leq k \leq l$, set $\qbinom{l}{k} = \frac{[l]!}{[k]![l-k]!}$.

\subsection{The unrolled quantum group of $\sltwo$}

We recall the definition of the unrolled quantum group of $\sltwo$, as introduced in \cite{geer_2009,costantino2015}. Pre-cursors of the unrolled quantum group appear in work of Ohtuski \cite{ohtsuki2002quantum}.
	
\begin{definition} \label{unrolled} 
The \emph{unrolled quantum group of $\mathfrak{sl}_2(\C)$} is the unital associative algebra $U^H_q(\sltwo)$ generated by $K$, $K^{-1}$, $H$, $E$ and $F$ with relations
\[
KK^{-1} = K^{-1}K = 1, \qquad HK = KH,
\]
\[
HE - EH = 2E,
\qquad
HF - FH = -2F,
\]
\[
KE = q^2EK,
\qquad
KF = q^{-2}FK,
\]
\[
EF - FE = \frac{K - K^{-1}}{q - q^{-1}}.
\]
The \emph{restricted unrolled quantum group} $\sunrolled$ is the quotient of $U^H_q (\sltwo)$ by the relations $E^r = F^r = 0$.
\end{definition}
	
Informally, the generator $H$ should be viewed as a logarithm of $K$. While this constraint is not imposed at the level of algebras, it is imposed on the modules of interest in this paper. See \cref{sec:weightMod} below.
		
Both $U^H_q (\sltwo)$ and $\sunrolled$ are Hopf algebras with coproduct, counit and antipode defined by
\begin{center}
	\setlength{\tabcolsep}{25pt}
	\begin{tabular}{l l l}
		$\Delta(E) = 1\otimes E + E\otimes K$, & $\varepsilon(E) = 0$, & $S(E) = -EK^{-1}$,\\
		$\Delta(F) = F\otimes 1 + K^{-1}\otimes F$, & $\varepsilon(F) = 0$, & $S(F) = -KF$, \\
		$\Delta(K) = K\otimes K$, & $\varepsilon(K) = 1$, & $S(K) = K^{-1}$, \\
		$\Delta(H) = H\otimes 1 + 1\otimes H$, & $\varepsilon(H) = 0$, & $S(H) = -H$.
	\end{tabular}
\end{center}
	
The De Concini--Kac quantum group $U_q(\sltwo)$ is isomorphic to the Hopf subalgebra of $U^H_q (\sltwo)$ generated by $E,F$ and $K^{\pm 1}$. Similarly, the restricted quantum group $\overline{U}_q(\sltwo)$ is isomorphic to the Hopf subalgebra of $\sunrolled$ generated by $E$, $F$ and $K^{\pm 1}$. The algebra $\sunrolled$ shares many properties with $\overline{U}_q(\sltwo)$. For example, $\sunrolled$ has a Poincar\'{e}--Birkhoff--Witt basis
\[
\{F^a H^b K^c E^d \mid 0 \leq a, d \leq r-1, \; b \in \Z_{\geq 0}, \; c \in \Z \}
\]
and admits a triangular decomposition
\[
\overline{U}^{H,-}_q(\sltwo) \otimes \overline{U}^{H,0}_q(\sltwo) \otimes \overline{U}^{H,+}_q(\sltwo) \xrightarrow[]{\sim} \sunrolled
\]
where $\overline{U}^{H,-}_q(\sltwo)$, $\overline{U}^{H,0}_q(\sltwo)$ and $\overline{U}^{H,+}_q(\sltwo)$ are the subalgebras of $\sunrolled$ generated by $F$, $H$ and $K^{\pm 1}$ and $E$, respectively. For later use, let $\Bunrolled$ be the Hopf subalgebra of $\sunrolled$ generated by $E$, $K^{\pm 1}$ and $H$. 

\subsection{Weight modules}\label{sec:weightMod}

Recall that all modules are assumed to be finite dimensional.
\begin{definition}\label{def:weightModule}
Let $V$ be a $\sunrolled$-module.
    \begin{enumerate}
        \item A \emph{weight vector of weight $\lambda \in \mathbb{C}$} is a non-zero vector $v \in V$ which satisfies $Hv = \lambda v$. If, moreover, $Ev=0$, then $v$ is called a \emph{highest weight vector}. The subspace $V[\lambda] = \{v \in V \mid Hv = \lambda v\}$ is called the \emph{weight space of weight $\lambda$}.
        \item \label{ite:KasH} The module $V$ is called a \emph{weight module} if it is the direct sum of its weight spaces, $V = \bigoplus_{\lambda \in \mathbb{C}} V[\lambda]$, and $Kv = q^{\lambda}v$ for all $v \in V[\lambda]$.
        \item The module $V$ is called a \emph{highest weight module} if it is generated by a highest weight vector. \qedhere
    \end{enumerate}
\end{definition}
	
All $\sunrolled$-modules considered in this paper are assumed to be weight modules. The second condition in \cref{def:weightModule}(\ref{ite:KasH}) can be written as the equality as operators $K = q^H$ on $V$. In view of this, when speaking of weight modules we often give the action of $H$ and omit that of $K$. Finally, note that a highest weight module is necessarily a weight module. 
	
Let $\mathcal{C}$ be the category of weight $\sunrolled$-modules and their $\sunrolled$-linear maps. The category $\mathcal{C}$ is $\mathbb{C}$-linear, locally finite and abelian. The bialgebra structure of $\sunrolled$ makes $\mathcal{C}$ into a monoidal category with unit object the one dimensional module $\C$ on which $H$, $E$ and $F$ act by zero. The associators and unitors are as for the category of complex vector spaces and are henceforth suppressed from the notation.
	
Let $V \in \mathcal{C}$. Denote by $V^{\vee} \in \mathcal{C}$ the dual vector space $\Hom_\mathbb{C}(V, \mathbb{C})$ with $\sunrolled$-module structure given by
\[
(x \cdot f)(v) = f(S(x)v), \qquad x \in \sunrolled, \qquad f \in V^{\vee}, \qquad v \in V.
\]
Given a basis $\{v_i\}_{i=1}^n$ of $V$ with dual basis $\{v_i^\vee\}_{i=1}^n$ of $V^\vee$, define
\[
\widehat{\ev}_V:V \otimes V^\vee \rightarrow \mathbb{C}, \qquad	v\otimes f \mapsto f(K^{1-r}v)
\]
and 
\[
\widehat{\coev}_V:\mathbb{C}\rightarrow V^\vee\otimes V, \qquad	1 \mapsto \sum_{i=1}^n K^{r-1}v_i^\vee\otimes v_i.
\]
Note that $\widehat{\coev}_V$ is independent of the choice of basis. A direct check shows that $\widehat{\ev}_V$ and $\widehat{\coev}_V$ are $\sunrolled$-linear and satisfy the snake relations, namely, that the compositions
\[
V \xrightarrow[]{\id_V \otimes \widehat{\coev}_V} V \otimes V^{\vee} \otimes V \xrightarrow{\widehat{\ev}_V \otimes \id_V} V
\]
and
\begin{equation}
\label{eq:snakeRel}
V^{\vee} \xrightarrow[]{\widehat{\coev}_V \otimes \id_{V^{\vee}}} V^{\vee} \otimes V \otimes V^{\vee} \xrightarrow{\id_{V^{\vee}} \otimes \widehat{\ev}_V} V^{\vee}
\end{equation}
are the respective identities. It follows that $\widehat{\ev}_V$ and $\widehat{\coev}_V$ are right duality morphisms. Define also
\[
\ev_V:V^\vee \otimes V \rightarrow \mathbb{C}, \qquad f\otimes v \mapsto f(v)
\]
and
\[
\coev_V:\mathbb{C}\rightarrow V\otimes V^\vee, \qquad 1 \mapsto \sum_{i=1}^n{v_i\otimes v_i^\vee}.
\]
These are the usual left duality morphisms in the category of finite dimensional vector spaces and are easily verified to be $\sunrolled$-linear. It follows that the category $\mathcal{C}$ is rigid. Hence, $\mathcal{C}$ is tensor in the sense of \cite[Definition 4.1.1]{etingof2016tensor}.

Given a finite dimensional vector space $V$, write $V \rightarrow V^{\vee \vee}$, $v \mapsto \spr{-}{v}$, for the canonical evaluation isomorphism.

\begin{lemma}
	The maps $\{p_V: V \rightarrow V^{\vee \vee}\}_{V \in \mathcal{C}}$ given by $p_V(v)= K^{1-r} \spr{-}{v}$ define a pivotal structure on $\mathcal{C}$.
\end{lemma}
\begin{proof}
    We need to verify that $\{p_V\}_{V \in \mathcal{C}}$ are the components of a monoidal natural isomorphism $p: \id_{\mathcal{C}} \Rightarrow (-)^{\vee} \circ (-)^{\vee}$. Naturality is immediate and a direct check shows that $p_V$ is $\sunrolled$-linear. The relation $\Delta(K^{1-r}) = K^{1-r} \otimes K^{1-r}$ implies the equality $p_{V \otimes W} = p_V \otimes p_W$, $V,W \in \mathcal{C}$, which is the required monoidality.
\end{proof}

One can readily see that the right and left duality structures defined above are compatible with the above pivotal structure, in the sense that the equalities $\id_{V^{\vee}} \otimes (p_V \circ \widehat{\coev}_V) = \coev_{V^{\vee}}$ and $\widehat{\ev}_V = (\ev_{V^{\vee}} \circ p_V) \otimes \id_{V^{\vee}}$ hold for each $V \in \mathcal{C}$.
    
\subsection{Simple modules}\label{sec:simpObj}

A non-zero module $V \in \mathcal{C}$ is called \emph{simple} (or \emph{irreducible}) if it has no non-zero proper submodules. In this section, we classify simple objects of $\mathcal{C}$. The results of this section are contained in  \cite[\S 5]{costantino2015}, although we give different proofs.
	
\begin{lemma}\label{highestweight}
	Every simple object of $\mathcal{C}$ is a highest weight module. 
\end{lemma}
\begin{proof}
   Let $V \in \mathcal{C}$ be simple and $v \in V$ a weight vector. Since $E^r=0$, there exists a minimal integer $l > 0$ such that $E^l v=0$. Then $E^{l-1} v$ is a highest weight vector and $\sunrolled \cdot E^{l-1} v \subset V$ is a non-zero submodule which, by simplicity, is equal to $V$.
\end{proof}

Let $\alpha \in \C$. Denote by $\mathbb{C}_{\alpha + r -1}$ the one dimensional weight $\Bunrolled$-module of $H$-weight $\alpha + r -1$ on which $E$ and $F$ act by zero.
	
\begin{definition}
	The \emph{Verma module of highest weight $\alpha+r-1$} is the $\sunrolled$-module
	$V_\alpha = \sunrolled \otimes_{\Bunrolled} \mathbb{C}_{\alpha + r - 1}.$
\end{definition}
	
Write $v_i$ for the vector $F^i \otimes 1 \in V_\alpha$. The Poincar\'{e}--Birkhoff--Witt basis for $\sunrolled$ shows that $\{v_0, \ldots, v_{r-1}\}$ \label{basis} is a weight basis of $V_\alpha$ and $V_{\alpha} \in \mathcal{C}$. Direct calculations show that the $\sunrolled$-action on $V_\alpha$ is given by
\[
H v_i = (\alpha + r - 1 - 2i) v_i, \quad Ev_i = \frac{\{i\}\{i -\alpha\}}{\{1\}^2}v_{i-1}, \quad Fv_i = v_{i+1},
\]
where by convention $v_{-1}=v_{r}=0$. In particular, $V_\alpha$ is a highest weight module generated by $v_0$. The structure of $V_{\alpha}$ is summarized by the diagram
\[
\begin{tikzcd}[column sep=4em]
	0 & v_{r-1} \ar[loop above, "H=\alpha - r + 1"] \ar[r, bend left, "E"] \ar[l, bend left, "F"below] & v_{r-2} \ar[loop above, "H=\alpha -r + 3"] \ar[l, bend left, "F"below] \ar[r, bend left, "E"above] & \ldots \ar[r, bend left, "E"above right] \ar[l, bend left, "F"below] & v_{1} \ar[loop above, "H=\alpha + r - 3"] \ar[r, bend left, "E"] \ar[l, bend left, "F"] & v_0 \ar[loop above, "H=\alpha + r -1"] \ar[r, bend left, "E"] \ar[l, bend left, "F"] & 0. \\
\end{tikzcd}
\]

\begin{lemma}\label{verma}
  If $V$ is a highest weight module of highest weight $\alpha + r - 1$, then there exists a surjection $V_\alpha \twoheadrightarrow V$.
\end{lemma}
\begin{proof}
By adjunction, there is an isomorphism
\[
\Hom_{\mathcal{C}}(V_{\alpha}, V) \simeq \Hom_{\Bunrolled}(\mathbb{C}_{\alpha + r - 1},V_{\big\vert \Bunrolled}).
\]
It follows that $\Hom_{\mathcal{C}}(V_{\alpha}, V)$ is isomorphic to the subspace of highest weight vectors of weight $\alpha+r-1$ in $V$. In particular, if $v \in V$ is a generating highest weight vector of weight $\alpha+r-1$, then the assignment $v_0 \mapsto v$ extends to a surjective morphism $V_{\alpha} \rightarrow V$ in $\mathcal{C}$.
\end{proof}

Using \cref{verma}, it is straightforward to verify that the map $v_{\alpha,r-1}^{\vee} \mapsto v_{-\alpha,0}$ extends to a $\sunrolled$-module isomorphism
\begin{equation}
\label{ValphaDual}
V_{\alpha}^{\vee} \xrightarrow[]{\sim} V_{-\alpha}.
\end{equation}
	
It follows from \cref{highestweight,verma} that any simple object of $\mathcal{C}$ is a quotient of a unique Verma module $V_\alpha$. In particular, a simple module has dimension at most $r$.
	
\begin{proposition} \label{simplemodules}
Let $\alpha \in \mathbb{C}$.
    \begin{enumerate}
        \item \label{eins} If $\alpha \notin \Z \setminus r \Z$, then $V_{\alpha}$ is simple.
        \item \label{zwei} If $\alpha \in \mathbb{Z} \setminus r\mathbb{Z}$ is written in its unique form as $\alpha = (l-1)r + n + 1$ with $0 \leq n \leq r-2$ and $l \in \mathbb{Z}$, then there exists a non-split short exact sequence $$ 0 \rightarrow S_{r-n-2}^{(l-1)r} \rightarrow V_\alpha \rightarrow S_{n}^{lr} \rightarrow 0$$ which is a Jordan--H\"{o}lder filtration of $V_{\alpha}$.
        \item \label{drei} Any simple object of $\mathcal{C}$ is isomorphic to a unique module of the form $V_\alpha$, $\alpha \in \C \setminus \Z \cup r \Z$, or $S_n^{lr}$, $l \in \Z$, $0 \leq n \leq r-2$.
    \end{enumerate} 
\end{proposition}
	
\begin{proof}
	If $\alpha \notin \Z \setminus r \Z$, then $\frac{\{i\}\{i - \alpha\}}{\{1\}^2} \neq 0$ for $i=1, \dots, r-1$, as follows from the assumption that $q$ is a primitive $2r$\textsuperscript{th} root of unity. It follows from the explicit form of the action of $E$ on $V_{\alpha}$ that $Ev_i \neq 0$ for $i=1, \dots, r-1$, whence $V_{\alpha}$ is simple. 
		
	If instead $\alpha \in \mathbb{Z} \setminus r\mathbb{Z}$, then $V_\alpha$ has exactly one proper submodule. Indeed, write $\alpha = (l-1)r + n + 1$ as in the statement of the proposition, so that $V_\alpha$ is of highest weight $lr + n$. Examining the action of $E$ on $V_\alpha$ shows that $Ev_{n+1} = 0$ and $Ev_i \neq 0$ if $i\neq 0, n+1$. Hence, $S := \Span\{v_{n+1}, \ldots, v_{r-1}\}$ is the unique proper submodule of $V_\alpha$. The module $S$ has dimension $r - n - 1$ and its quotient $S_n^{lr}:=V_\alpha / S$ is a simple highest weight module of highest weight $lr + n$ and dimension $n+1$. By \cref{verma}, there exists a surjection $V_{(l-1)r-n-1} \rightarrow S$ which, by the argument of this paragraph, descends to an isomorphism $S_{r - n - 2}^{(l-1)r} \xrightarrow[]{\sim}S$. Finally, the uniqueness of $S$ implies that the sequence $$ 0 \rightarrow S_{r-n-2}^{(l-1)r} \rightarrow V_\alpha \rightarrow S_{n}^{lr} \rightarrow 0$$ is non-split.
		
	By \cref{highestweight,verma}, any simple module is a quotient of a Verma module. Thus, the third statement of the proposition follows from the first two.
\end{proof}
	
\begin{remark}
    Since $\overline{U}_q(\sltwo)$ is a Hopf subalgebra of $\sunrolled$, there is a monoidal forgetful functor $\mathcal{C} \rightarrow \overline{U}_q(\sltwo) \lmod$. In the notation of \cite[\S 2.11]{jantzen}, this functor sends the simple objects $S^{lr}_n$ and $V_{\alpha}$ of $\mathcal{C}$ to $L(n,(-1)^l)$ and $Z_0(q^{\alpha+r-1})$, respectively.
\end{remark}

\cref{simplemodules} implies that a simple object is determined up to isomorphism by its highest weight and that the simple objects $S^{lr}_n$ are neither injective nor projective.
	
\begin{proposition}\label{valphaproj}
    If $\alpha \in \mathbb{C} \setminus \Z \cup r \Z$, then $V_{\alpha} \in \mathcal{C}$ is projective and injective.
\end{proposition}
\begin{proof}
Let $f: V \twoheadrightarrow W$ be a surjection in $\mathcal{C}$ and $\phi: V_{\alpha} \rightarrow W$ a non-zero morphism. By the proof of \cref{verma}, the map $\phi$ is determined by a highest weight vector $\phi(v_0)=w \in W$ of weight $\alpha+r-1$. Surjectivity of $f$ implies that $w$ has a preimage under $f$, say $v$, which is of weight $\alpha+r-1$ and satisfies $Ev \in \ker f$.

Let $\xi = Ev$, which is of weight $\alpha+r+1$ and satisfies $E^{r-1} \xi=0$. For any $a_0, \dots, a_{r-2} \in \mathbb{C}$, the vector
\[
v^{\prime} = v + \sum_{i=0}^{r-2} a_i F^{i+1}E^i \xi.
\]
is of weight $\alpha+r-1$ and satisfies $f(v^{\prime})=w$. Using \cite[\S 1.3]{jantzen}, we compute
\[
E v^{\prime}
=
\xi + \sum_{i=0}^{r-2} a_i(F^{i+1}E^{i+1} \xi + [i+1] [\alpha+r+1+i] F^i E^i \xi).
\]
Then $Ev^{\prime}=0$ if and only if the recursive equations
\[
a_i = - [i+2][\alpha+r+2+i] a_{i+1} \qquad i=-1, \dots, r-3,
\]
hold, with $a_{-1}=1$. This recursive system determines $\{a_i\}_i$ if and only if $\alpha \in \mathbb{C} \setminus \Z \cup r \Z$, as otherwise the coefficient of some $a_{i+1}$ vanishes. Arguing as in the start of the proof, the assignment $v_0 \mapsto v^{\prime}$ determines a $\sunrolled$-linear map $\tilde{\phi}:V_{\alpha} \rightarrow V$ which satisfies $f \circ \tilde{\phi} = \phi$. This establishes the projectivity of $V_{\alpha}$.

In view of the isomorphism \eqref{ValphaDual} and the previous paragraph, the module $V_\alpha^\vee$ is projective. Standard adjunction isomorphisms (see \cite[Proposition 2.10.8]{etingof2016tensor}) give a natural isomorphism of contravariant functors
\[
\Hom_{\mathcal{C}}(-,V_{\alpha}) \simeq \Hom_{\mathcal{C}}(V_{\alpha}^{\vee}, -) \circ (-)^{\vee}.
\]
Because $V_\alpha^\vee$ is projective, $\Hom_\mathcal{C}(V_{\alpha}^{\vee}, -)$ is an exact functor. Because $(-)^\vee$ is an exact functor at the level of complex vector spaces, it is also exact on $\mathcal{C}$. Hence, the functor $\Hom_{\mathcal{C}}(-,V_{\alpha})$ is exact and $V_\alpha$ is injective.
\end{proof}

\subsection{Generic semisimplicity}
	
Recall that an abelian category is called \emph{semisimple} if every object is a direct sum of simple objects. In view of \cref{simplemodules}(\ref{zwei}), the category $\mathcal{C}$ is not semisimple. However, $\mathcal{C}$ fails to be semisimple in a controlled manner. The goal of this section is to make this statement precise. To do so, we begin with some general definitions from \cite{geer_2017}.

Let $G$ be an additive abelian group.
	
\begin{definition}
	A \emph{$G$-grading} on a rigid monoidal category $\mathcal{D}$ is the data of non-empty full subcategories $\mathcal{D}_g \subset \mathcal{D}$, $g \in G$, such that $\mathcal{D} = \bigoplus_{g \in G} \mathcal{D}_g$ and $V^\vee \in \mathcal{D}_{-g}$ and $V \otimes V' \in \mathcal{D}_{g+g'}$ whenever $V \in \mathcal{D}_g$ and $V' \in \mathcal{D}_{g'}$.
\end{definition}
	
\begin{definition}
	A subset $X \subset G$ is called \emph{symmetric} if $-X = X$ and \emph{small} if $G \neq \bigcup_{i=1}^n (g_i+ X)$ for all $g_1, \ldots, g_n \in G$.
\end{definition}

\begin{definition}
	A $G$-graded category $\mathcal{D}$ is called \emph{generically semisimple with small symmetric subset $X \subset G$} if $\mathcal{D}_g$ is semisimple whenever $g \in G\setminus X$. In this case, a simple module $V \in \mathcal{C}_g$ in degree $g \in G \setminus X$ is called \emph{generic simple}.
\end{definition}

Consider again the category of weight modules over $\sunrolled$. Let $G$ be the additive group $\mathbb{C}/2\Z$. For each $ \overline{\alpha} \in \mathbb{C} \slash 2\Z$, let $\mathcal{C}_{\overline{\alpha}}$ be the full subcategory of $\mathcal{C}$ consisting of modules whose weights are in the class $\overline{\alpha}$. The Hopf algebra structure of $\sunrolled$ shows that $\mathcal{C} = \bigoplus_{\overline{\alpha} \in \mathbb{C} \slash 2 \Z} \mathcal{C}_{\overline{\alpha}}$ is a $\C \slash 2 \Z$-grading.

\begin{theorem} \label{Cisgss}
	The $\C/2\Z$-graded category $\mathcal{C}$ is generically semisimple with small symmetric subset $\Z \slash 2 \Z \subset \C \slash 2\Z$.
\end{theorem}
\begin{proof}
	Let $\overline{\alpha} \in (\mathbb{C}/2\Z) \setminus (\Z/2\Z)$ and $V \in \mathcal{C}_{\overline{\alpha}}$ a non-zero object. Then $V$ contains a highest weight vector $v$ of weight $\alpha \in \C$, where $\alpha$ is in the class $\overline{\alpha}$. The assumption on $\overline{\alpha}$ implies that the submodule generated by $v$ is isomorphic to $V_{\alpha - r + 1}$; see the proof of \cref{simplemodules}. The module $V_{\alpha - r + 1}$ is injective by \cref{valphaproj}, whence there is a splitting $V \simeq V_{\alpha - r + 1} \oplus V^{\prime}$ for some $V^{\prime} \in \mathcal{C}_{\overline{\alpha}}$ of dimension strictly less than that of $V$. An induction argument on the dimension of $V$ then completes the proof.
\end{proof}

In view of \cref{simplemodules}, the generic simple objects of $\mathcal{C}$ are the Verma modules $V_\alpha$ with $\alpha \in \C \setminus \Z$.

\subsection{Braiding}\label{sec:braiding}
	
In this section, we construct a braiding on $\mathcal{C}$. The form of the braiding is motivated by the universal $R$-matrix for the $\hbar$-adic quantum group $U_{\hbar}(\sltwo)$, as described in \cite[\S 10]{drinfeld1986quantum}, \cite[\S\S 4.5 and A.2]{ohtsuki2002quantum}.
	
\begin{definition}
	The \emph{$r$-truncated q-exponential map} is $\exp_q^<(x) = \sum_{l=0}^{r-1} \frac{q^{l (l-1) / 2}}{[l]!} x^l$.
\end{definition}
	
Let $V, W \in \mathcal{C}$ with weight bases $\{v_i\}_i$ and $\{w_j\}_j$ of weights $\{\lambda_i^v\}_i$ and $\{\lambda_j^w\}_j$, respectively. Define $q^{H\otimes H/2} \in \End_{\mathbb{C}}(V \otimes W)$ by 
\[
q^{H\otimes H/2}(v_i\otimes w_j) = q^{\lambda_i^v \lambda_j^w/2}v_i\otimes w_j
\]
and $R \in \End_{\mathbb{C}}(V \otimes W)$ as
\[
R = q^{H \otimes H / 2} \circ \exp_q^<(\{1\} E \otimes F) = q^{H\otimes H/2} \circ \sum_{l=0}^{r-1}\frac{\{1\}^{2l}}{\{l\}!}q^{l(l-1)/2}E^l\otimes F^l,
\]
where $\exp_q^<(\{1\} E \otimes F)$ is viewed as a $\C$-linear map via left multiplication. Finally, define $c_{V,W} \in \Hom_{\mathbb{C}}(V \otimes W, W \otimes V)$ as
\[
c_{V,W}(v\otimes w) = \tau R(v\otimes w),
\]
where $\tau$ is the swap map $V \otimes W \rightarrow W \otimes V$, $v \otimes w \mapsto w \otimes v$.

\begin{lemma}\label{homomorphism}
	The map $c_{V,W}$ is $\sunrolled$-linear.
\end{lemma}
\begin{proof}
	It suffices to check linearity of $c_{V,W}$ on the generators $H,F,E$. Let $v \in V$ and $w \in W$ of weight $\lambda^v$ and $\lambda^w$, respectively. We have
	\begin{align*}
		H \cdot c_{V,W}(v \otimes w) &= \tau q^{H\otimes H/2}\sum_{l=0}^{r-1}\frac{q^{l (l-1) / 2}}{[l]!}E^l\otimes F^l ((H+2l)v \otimes w + v \otimes (H-2l)w)\\
		&= (\lambda^v + \lambda^w) c_{V,W}(v \otimes w)\\
		&= c_{V,W}(H \cdot v \otimes w).
	\end{align*}
	We prove $E$-linearity. We have an equality
	\[
	K \otimes E \circ q^{H \otimes H / 2} = q^{H \otimes H / 2} \circ 1 \otimes E
	\]
	in $\End_{\C}(V \otimes W)$. Indeed, we compute
	\[
	K \otimes E \circ q^{H \otimes H / 2} (v_i \otimes w_j) = q^{\lambda^v_i}q^{\lambda^v_i \lambda^w_j / 2} v_i \otimes E w_{j} = q^{\lambda^v_i (\lambda^w_j +2) / 2} v_i \otimes E w_{j}
	\]
	and
	\[
	q^{H \otimes H / 2} \circ 1 \otimes E (v_i \otimes w_j) = q^{H \otimes H / 2} v_i \otimes E w_{j} = q^{\lambda^v_i (\lambda^w_j +2) / 2} v_i \otimes E w_{j}.
	\]
	A similar calculation shows that $E \otimes 1 \circ q^{H \otimes H / 2} = q^{H \otimes H / 2} \circ E \otimes K^{-1}$. Using these two equalities, $E$-linearity of $c_{V,W}$ reduces to the equality
	\[
	(E \otimes K^{-1} + 1 \otimes E) \exp_q^<(\{1\} E \otimes F) = \exp_q^<(\{1\} E \otimes F) (E \otimes K + 1 \otimes E)
	\]
	in $\sunrolled \otimes \sunrolled$, which is proved in \cite[Equation A.10]{ohtsuki2002quantum}. Linearity of $F$ is proved similarly.
\end{proof}
	
\begin{lemma}\label{invertible}
	The map $c_{V,W}$ is invertible.
\end{lemma}
\begin{proof}
	Clearly $\tau$ is invertible. We claim that the inverse of $R$ is
	\[
	R^{-1} = \exp_{q^{-1}}^< (-\{1\} E \otimes F) q^{-H \otimes H /2}.
	\]
    Compare with \cite[\S A.2]{ohtsuki2002quantum}.
	We have $q^{H \otimes H /2} \circ q^{-H \otimes H /2}= 1$.
	By definition,
	\begin{multline*}
		\exp_q^<(\{1\} E \otimes F) \cdot \exp_{q^{-1}}^<(-\{1\} E \otimes F) \\ = \sum_{l=0}^{r-1} \sum_{k=0}^{r-1} \frac{q^{l (l-1) / 2} q^{-k (k-1) / 2}}{[l]![k]!} (-1)^k (\{1\} E \otimes F)^{l+k}.
	\end{multline*}
	Since $(E\otimes F)^r=0$, the double sum is
	\[
    \sum_{i=0}^{r-1} \frac{q^{-i (i-1) / 2}}{[i]!} (-\{1\} E \otimes F)^{i} \sum_{l=0}^{i} (-1)^l \qbinom{i}{l} q^{l (i-1)} .
	\]
	The sum $\sum_{l=0}^{i} (-1)^l \qbinom{i}{l} q^{l (i-1)}$ is $0$ for $i > 0$ and $1$ if $i=0$; see \cite[\S 0.2]{jantzen}. The inverse of $R$ is thus as stated.
\end{proof}
	
\begin{proposition}\label{ribbon}
	The maps $\{c_{V,W}:V\otimes W \rightarrow W\otimes V\}_{V,W}$ define a braiding on $\mathcal{C}$.
\end{proposition}
\begin{proof}
	\cref{homomorphism,invertible} yield that the maps $c_{V,W}$ give a family of isomorphisms in $\mathcal{C}$. Naturality of $c_{V,W}$  follows from the fact that the endomorphism $q^{H \otimes H /2}$ commutes with $\sunrolled$-linear maps. 
	It remains to verify the hexagon axioms \cite[Definition 8.1.1]{etingof2016tensor}. Let $V,W,U \in \mathcal{C}$. We prove that
	\begin{equation}
		c_{V, W \otimes U} = (\id_W \otimes c_{V,U})\circ (c_{V,W} \otimes \id_U)
	\end{equation}
	and leave the verification of the equality $c_{V \otimes W, U} = (c_{V,U} \otimes \id_W)\circ (\id_V \otimes c_{W,U})$ to the reader. Let $v \in V, w \in W, u \in U$ be weight vectors of weights $\lambda^v,\lambda^w$, $\lambda^u$, respectively.
	We compute
	\[ 
	c_{V, W \otimes U}(v\otimes w \otimes u) = q^{H \otimes H / 2} \sum_{l=0}^{r-1} \frac{\{1\}^{2l}}{\{l\}!}q^{l(l-1)/2} \Delta(F^l) (w \otimes u) \otimes E^lv. 
	\]
	A straightforward induction argument shows that
	\[
	\Delta(F^l) = \sum_{i=0}^l q^{i(l-i)} \qbinom{l}{i} F^i K^{-(l-i)} \otimes F^{l-i}.
	\]
	Compare with \cite[\S 3.1]{jantzen}, where slightly different conventions are used. Using this, we find
	\[
	c_{V, W \otimes U}(v\otimes w \otimes u) = q^{H \otimes H / 2} \sum_{l=0}^{r-1} \frac{\{1\}^{2l}}{\{l\}!}q^{l(l-1)/2} \sum_{i=0}^l q^{(i- \lambda^w)(l-i)}\qbinom{l}{i} F^i \otimes F^{l-i} \otimes E^l (w \otimes u \otimes v)
	\]
    which evaluates to
	\[
	\sum_{l=0}^{r-1} \frac{\{1\}^{2l}}{\{l\}!}q^{l(l-1)/2} \sum_{i=0}^l q^{(\lambda^w + \lambda^u -2l) (\lambda^v + 2l)/2} q^{(i- \lambda^w)(l-i)}\qbinom{l}{i} F^i \otimes F^{l-i} \otimes E^l (w \otimes u \otimes v).
	\]
	On the other hand, we compute
	\begin{align*}
		c_{V,W} \otimes \id_U (v\otimes w \otimes u) &=  q^{H \otimes H /2} \sum_{l=0}^{r-1} \frac{\{1\}^{2l}}{\{l\}!}q^{l(l-1)/2} F^n w \otimes E^n v\otimes u \\
		&=  \sum_{l=0}^{r-1} \frac{\{1\}^{2l}}{\{l\}!}q^{l(l-1)/2} q^{(\lambda^w - 2l)(\lambda^v + 2l)/2} F^l w \otimes E^l v\otimes u.
	\end{align*}
	Applying $\id_W \otimes c_{V,U}$ then gives
	\begin{multline*}
		\id_W \otimes c_{V,U} (\sum_{l=0}^{r-1} \frac{\{1\}^{2l}}{\{l\}!}q^{l(l-1)/2} q^{(\lambda^w - 2l)(\lambda^v + 2l)/2} F^l w \otimes E^l v\otimes u) \\
	    =\sum_{l=0}^{r-1} \sum_{k=0}^{r-1} \frac{\{1\}^{2l +2k}}{\{l\}!\{k\}!} q^{l(l-1)/2}q^{k(k-1)/2} q^{(\lambda^w - 2l)(\lambda^v + 2l)/2} \cdot \\ \cdot q^{(\lambda^u - 2k)(\lambda^v + 2l + 2k)/2} \cdot (F^l w \otimes F^k u \otimes E^k E^l v).
	\end{multline*}
	To check the equality $c_{V, W \otimes U}(v\otimes w \otimes u) = c_{V,W} \otimes \id_U (v\otimes w \otimes u)$, we compare the coefficients of $F^a w \otimes F^b u \otimes E^{a+b} v$. The coefficients on the left and right-hand sides of the desired equality are
	\[
	\frac{\{1\}^{2(a+b)}}{\{a+b\}!}q^{(a+b)((a+b)-1)/2} q^{(\lambda^w + \lambda^u -2(a+b)) (\lambda^v + 2(a+b))/2} q^{(a- \lambda^w)((a+b)-a)}\qbinom{a+b}{a}
	\]
	and
	\[
	\frac{\{1\}^{2a +2b}}{\{a\}!\{b\}!} q^{a(a-1)/2} q^{b(b-1)/2} q^{(\lambda^w - 2a)(\lambda^v + 2a)/2} q^{(\lambda^u - 2b)(\lambda^v + 2a + 2b)/2},
	\]
	respectively, which are equal by direct verification. \qedhere
\end{proof}

\subsection{Ribbon structure}
	
In this section, we construct a ribbon structure on $\mathcal{C}$. Having already established that $\mathcal{C}$ is braided (\cref{ribbon}), a ribbon structure is the additional data of a \emph{twist}, that is, a natural automorphism of the identity functor $\theta: \id_{\mathcal{C}} \Rightarrow \id_{\mathcal{C}}$ which satisfies the \emph{balancing condition}
$$\theta_{V \otimes W} = (\theta_V \otimes \theta_W) \circ c_{W,V} \circ c_{V,W}$$ and the \emph{ribbon condition} 	
\begin{equation}\label{ribCond}
    (\theta_V)^\vee = \theta_{V^\vee}
\end{equation}
for all $V,W \in \mathcal{C}$
	
Recall that the right partial trace of $f \in \End_{\mathcal{C}}(V \otimes W)$ is the endomorphism $\ptr_R(f) \in \End_{\mathcal{C}}(V)$ defined by
\[
V \xrightarrow{\id_V \otimes \coev_W} V \otimes W \otimes W^\vee \xrightarrow{f \otimes \id_{W^\vee}} V \otimes W \otimes W^\vee \xrightarrow{\id_V \otimes \widehat{\ev}_W} V.
\]
Define a natural automorphism $\theta: \id_{\mathcal{C}} \Rightarrow \id_{\mathcal{C}}$ by
\begin{equation} \label{eq:defTheta}
	\theta_V := \ptr_R(c_{V,V}), \qquad V \in \mathcal{C}
\end{equation}
where $c$ is the braiding of $\mathcal{C}$. The hexagon axioms of the braiding ensure that $\theta$ satisfies the balancing condition. To verify that $\theta$ also satisfies the ribbon condition, we use the following generic extension result.
	
\begin{theorem}[{\cite[Theorem 9]{geer_2017}}] \label{extendTwist}
	Let $\mathcal{D}$ be a generically semisimple pivotal braided category. Define a natural automorphism $\theta: \id_{\mathcal{D}} \Rightarrow \id_{\mathcal{D}}$ so that its components are given by Equation \eqref{eq:defTheta}. If $\theta_V^{\vee} = \theta_{V^{\vee}}$ for any generic simple object $V \in \mathcal{D}$, then $\theta$ is a twist on $\mathcal{D}$.
\end{theorem}

We can now prove the main result of this section.

\begin{theorem}\label{ribbon theorem}
	The natural transformations $c$ and $\theta$ equip $\mathcal{C}$ with the structure of a ribbon category.
\end{theorem}
	
\begin{proof}
	Recall that the generic simple objects of $\mathcal{C}$ are the Verma modules $V_{\alpha}$ with $\alpha \in \C \setminus \Z$. For any $\alpha \in \C$ the right partial trace of $c_{V_\alpha, V_\alpha}$ is
    \[
	V_\alpha \xrightarrow{\id_{V_\alpha} \otimes \coev_{V_\alpha}} V_\alpha \otimes V_\alpha \otimes V_\alpha^\vee \xrightarrow{c_{V_\alpha, V_\alpha} \otimes \id_{V_\alpha^\vee}} V_\alpha \otimes V_\alpha \otimes V_\alpha^\vee \xrightarrow{\id_{V_{\alpha}} \otimes \widehat{\ev}_{V_\alpha}} V_\alpha.
	\]
	Since $\End_{\mathcal{C}}(V_{\alpha}) \simeq \C$ (see the proof of \cref{verma}), it suffices to compute the image of the highest weight vector $v_{0} \in V_{\alpha}$ under this composition. We have
	\begin{multline*}
	    v_{0} \mapsto \sum_{i=0}^{r-1} v_{0} \otimes v_i \otimes v_i^{\vee}  \mapsto \sum_{i=0}^{r-1} q^{(\alpha + r - 1) (\alpha + r - 1 - 2i) /2} v_i \otimes v_0 \otimes v_i^{\vee} \\
	    \mapsto q^{(\alpha + r - 1) (\alpha + r - 1) /2} q^{(\alpha + r - 1) (1-r)} v_0
	    = q^{(\alpha + r - 1) (\alpha - r + 1) /2} v_0.
	\end{multline*}
	Since $V^\vee_\alpha \simeq V_{-\alpha}$ and the scalar $q^{(\alpha + r - 1) (\alpha - r + 1) /2}$ is unchanged under the substitution $\alpha \mapsto - \alpha$, it follows that $\theta_{V_{\alpha}^{\vee}} = \theta_{V_{\alpha}}^{\vee}$ for all $\alpha \in \C$. \cref{extendTwist} therefore applies in the present setting, allowing the conclusion that the maps $\{\theta_V\}_{V \in \mathcal{C}}$ define a twist on $\mathcal{C}$.
\end{proof}
 
\section{Reshetikhin--Turaev invariants} \label{rti}
	
We recall basic background material on Reshetikhin--Turaev invariants of links \cite{reshetikhin1990ribbon}. For a detailed introduction to the theory, the reader is referred to \cite{turaev2016quantum}. We end this section by modifying the Reshetikhin--Turaev construction to produce a non-zero invariant for knots colored by simple objects of vanishing quantum dimension. Readers who are well-versed in Reshetikhin--Turaev theory could remind themselves of Lemma \ref{cutting} and proceed to \cref{mqi}.
	
\subsection{Reshetikhin--Turaev invariants of links}
	
Let $\mathcal{D}$ be a ribbon category. Associated to $\mathcal{D}$ is the ribbon category of $\mathcal{D}$-colored ribbon graphs $\mathbf{Rib}_\mathcal{D}$ \cite[\S I.I.2]{turaev2016quantum}. Objects of $\mathbf{Rib}_\mathcal{D}$ are finite sequences of pairs $(V, \epsilon)$, where $V \in \mathcal{D}$ and $\epsilon \in \{\pm \}$. Morphisms in $\mathbf{Rib}_\mathcal{D}$ are isotopy classes of $\mathcal{D}$-colored ribbon graphs bordering two such sequences of objects. The colorings of the ribbon graphs are required to be compatible with the domain and codomain objects in the obvious sense. Composition of morphisms is defined by concatenation of ribbon graphs. The monoidal structure of $\mathbf{Rib}_\mathcal{D}$ is defined on objects by concatenation of sequences and on morphisms by disjoint union.
	
\begin{theorem}[{\cite[Theorem 2.5]{turaev2016quantum}}]\label{reshturfunctor}
     There exists a unique ribbon functor $F_{\mathcal{D}}: \mathbf{Rib}_\mathcal{D} \rightarrow \mathcal{D}$ such that $F_{\mathcal{D}}(V,+)=V$ and $F_{\mathcal{D}}(V,-) =V^\vee$ for all $V \in \mathcal{D}$. 
\end{theorem}
	
The functor $F_{\mathcal{D}}$ is called the \emph{Reshetikhin--Turaev functor}. The precise definition of the ribbon structure of $\mathbf{Rib}_\mathcal{D}$ and the fact that $F_{\mathcal{D}}$ is ribbon implies that $F_{\mathcal{D}}$ takes the following values on morphisms in $\mathbf{Rib}_\mathcal{D}$:
\begin{equation*} \label{eq:7}
	F_{\mathcal{D}} \left(
	\begin{tikzpicture}[anchorbase]
		\draw[->,very thick] (0,0) -- node[left] {$V$} (0,1);
	\end{tikzpicture} \right )
	\;\;\;   =
	\begin{tikzpicture}[anchorbase]
		\draw[->] (0,0) node[below] {$V$} -- node[left] {$\id_V$} (0,1) node[above] {$V$};
	\end{tikzpicture}
	\qquad \qquad \qquad
	F_{\mathcal{D}}\left(
	\begin{tikzpicture}[anchorbase]
		\draw [->, very thick] (.8,0) to [out=270,in=0] (.6,-.3) to [out=180,in=320] (.15,.1) to [out=140,in=270] (0,.8);
		\draw [very thick,cross line] (0,-.7) to [out=90,in=180] (.6,.3) to [out=0,in=90] (.8,0);
		\node at (.3,-.7) {$V$};
	\end{tikzpicture}
	\right)  =
	\begin{tikzpicture}[anchorbase]
		\draw[->] (0,0) node[below] {$V$} -- node[left] {$\theta_V$} (0,1) node[above] {$V$};
	\end{tikzpicture}
\end{equation*}
\begin{equation*}
	F_{\mathcal{D}}\left(
	\begin{tikzpicture}[anchorbase]
    	\draw[->,very thick] (1,0) -- node[right,near start] {$W$} (0,1);
		\draw[->,very thick,cross line] (0,0) -- node[left,near start] {$V$} (1,1);
	\end{tikzpicture}
	\right)  =
	\begin{tikzpicture}[anchorbase]
        \draw[->] (0,0) node[below] {$V \otimes W$} -- node[left] {$c_{V,W}$} (0,1) node[above] {$W \otimes V$};
	\end{tikzpicture}
	\qquad \qquad
	F_{\mathcal{D}}\left(
	\begin{tikzpicture}[anchorbase]
		\draw[->,very thick] (0,0) -- node[left,near start] {$W$} (1,1);
		\draw[->,very thick,cross line] (1,0) -- node[right,near start] {$V$} (0,1);
	\end{tikzpicture}
	\right)  =
	\begin{tikzpicture}[anchorbase]
		\draw[->] (0,0) node[below] {$W \otimes V$} -- node[left] {$c^{-1}_{V,W}$} (0,1) node[above] {$V \otimes W$};
	\end{tikzpicture}
\end{equation*}
	
\begin{equation*}
	F_{\mathcal{D}}\left(
	\begin{tikzpicture}[anchorbase]
		\draw[->,very thick] (0,0)  arc (0:180:0.5 and 0.75);
		\node at (-.3,.1) {$V$};
	\end{tikzpicture}
	\right)  =
	\begin{tikzpicture}[anchorbase]
		\draw[->] (0,0) node[below] {$V^\vee \otimes V$} -- node[left] {$\ev_V$} (0,1) node[above] {$\C$};
	\end{tikzpicture}
	\qquad\qquad
	F_\mathcal{D}\left(
	\begin{tikzpicture}[anchorbase]
		\draw[<-,very thick] (0,0)  arc (180:360:0.5 and 0.75);
		\node at (.3,0) {$V$};
	\end{tikzpicture}
	\right)  =
    \begin{tikzpicture}[anchorbase]
		\draw[->] (0,0) node[below] {$\C$} -- node[left] {$\coev_V$} (0,1) node[above] {$V \otimes V^\vee$};
	\end{tikzpicture}
    \end{equation*}
    \begin{equation*}
	F_{\mathcal{D}}\left(
	\begin{tikzpicture}[anchorbase]
		\draw[<-,very thick] (0,0)  arc (0:180:0.5 and 0.75);
		\node at (-.7,.1) {$V$};
	\end{tikzpicture}
	\right)  =
	\begin{tikzpicture}[anchorbase]
		\draw[->] (0,0) node[below] {$V \otimes V^\vee$} -- node[left] {$\widehat{\ev}_V$} (0,1) node[above] {$\C$};
	\end{tikzpicture}
	\qquad\qquad
	F_\mathcal{D}\left(
	\begin{tikzpicture}[anchorbase]
		\draw[->,very thick] (0,0)  arc (180:360:0.5 and 0.75);
		\node at (.7,0) {$V$};
	\end{tikzpicture}
	\right)  =
	\begin{tikzpicture}[anchorbase]
		\draw[->] (0,0) node[below] {$\C$} -- node[left] {$\widehat{\coev}_V$} (0,1) node[above] {$V^\vee \otimes V$};
	\end{tikzpicture}
	.
\end{equation*}
The above eight morphisms in $\mathbf{Rib}_\mathcal{D}$ generate all morphisms of $\mathbf{Rib}_\mathcal{D}$ \cite[\S I.3-4]{turaev2016quantum}. In particular, the value $F_{\mathcal{D}}$ on any morphism of $\mathbf{Rib}_\mathcal{D}$ can be computed as an iterated composition of (co)evaluations, (inverse) braidings and (inverse) twists in $\mathcal{D}$. Colored framed links are particular examples of morphisms in $\mathbf{Rib}_\mathcal{D}$- they are endomorphisms of the empty sequence. Thus, the assignment $L \mapsto \langle F_{\mathcal{D}} (L) \rangle$ is a isotopy invariant of colored framed links.
	
We record the following result which will be used below.
	
\begin{lemma} \label{rotate}
	For any $V, W \in \mathcal{D}$, the following equality of morphisms in $\mathcal{D}$ holds:
	\begin{equation}\label{rotate_loop}
		\begin{aligned}
			F_\mathcal{D} \left (
			\begin{tikzpicture}[anchorbase,scale=0.8]
				\draw [very thick, postaction={decorate}, decoration={markings, mark=at position 0.4 with {\arrow{>}}}] (-.1,-.5) to [out=180,in=270] (-.8,0) to [out=90,in=180] (0,.5) to [out=0,in=90] (.8,0) to [out=270,in=0] (.1,-.5);
				\draw [very thick] (0,-1) to (0,.4);
				\draw [->, very thick] (0,.6) to (0,1);
				\node at (1.1,0) {$W$};
				\node at (.25,-.9) {$V$};
			\end{tikzpicture} \right )
			\,=\, F_\mathcal{D} \left (
			\begin{tikzpicture}[anchorbase,scale=0.8]
				\draw [very thick, postaction={decorate}, decoration={markings, mark=at position 0.65 with {\arrow{<}}}] (.1,.5) to [out=0,in=90] (.8,0) to [out=270,in=0] (0,-.5) to [out=180,in=270] (-.8,0) to [out=90,in=180] (-.1,.5);
				\draw [very thick] (0,-1) to (0,-.6);
				\draw [->, very thick] (0,-.4) to (0,1);
				\node at (1.1,0) {$W$};
				\node at (.25,-.9) {$V$};
			\end{tikzpicture} \right ).
		\end{aligned}
	\end{equation}
 Moreover, if $V$ is simple, then the following equality of scalars holds:
 \begin{equation} \label{rotate_diagram}
		\begin{aligned}
			\left \langle F_\mathcal{D} \left (\begin{tikzpicture}[anchorbase,scale=0.8]
				\draw [very thick, postaction={decorate}, decoration={markings, mark=at position 0.4 with {\arrow{>}}}] (-.1,-.5) to [out=180,in=270] (-.8,0) to [out=90,in=180] (0,.5) to [out=0,in=90] (.8,0) to [out=270,in=0] (.1,-.5);
				\draw [very thick] (0,-1) to (0,.4);
				\draw [->, very thick] (0,.6) to (0,1);
				\node at (1.1,0) {$W$};
				\node at (.25,-.9) {$V$};
			\end{tikzpicture}
			\right ) \right \rangle
			\,=\, \left \langle F_\mathcal{D} \left ( \begin{tikzpicture}[anchorbase,scale=0.8]
				\draw [very thick, postaction={decorate}, decoration={markings, mark=at position 0.4 with {\arrow{>}}}] (.1,.5) to [out=0,in=90] (.8,0) to [out=270,in=0] (0,-.5) to [out=180,in=270] (-.8,0) to [out=90,in=180] (-.1,.5);
				\draw [<-,very thick] (0,-1) to (0,-.6);
				\draw [very thick] (0,-.4) to (0,1);
				\node at (1.1,0) {$W$};
				\node at (.4,-.9) {$V$};
			\end{tikzpicture} \right ) \right \rangle.
		\end{aligned}
	\end{equation}
\end{lemma}

\begin{proof}
	\cref{rotate_loop} holds by the following indicated combination of framed Reidemeister moves:
	\begin{multline*}
			F_\mathcal{D} \left (
			\begin{tikzpicture}[anchorbase,scale=0.8]
				\draw [very thick, postaction={decorate}, decoration={markings, mark=at position 0.65 with {\arrow{<}}}] (.1,.5) to [out=0,in=90] (.8,0) to [out=270,in=0] (0,-.5) to [out=180,in=270] (-.8,0) to [out=90,in=180] (-.1,.5);
				\draw [very thick] (0,-1) to (0,-.6);
				\draw [->, very thick] (0,-.4) to (0,1);
				\node at (1.1,0) {$W$};
				\node at (.25,-.9) {$V$};
			\end{tikzpicture} \right )
			\overset{\RI}{=}
			F_\mathcal{D} \left (
			\begin{tikzpicture}[anchorbase,scale=0.8]
				\draw [very thick] (.1,.5) to [out=0,in=90] (.8,0) to (.8,-.4);
				\draw [very thick] (-.1,.5) to [out=180,in=90] (-.8,0) to (-.8,-.4);
				\draw [very thick, postaction={decorate}, decoration={markings, mark=at position 0.4 with {\arrow{>}}}] (-.8,-.6) to [out=270,in=0] (-1,-.8) to [out=180,in=270] (-1.2,-.6) to [out=90,in=180] (-1,-.5) to (1,-.5) to [out=0,in=90] (1.2,-.6) to [out=270,in=0] (1,-.8) to [out=180,in=270] (.8,-.6);
				\draw [very thick] (0,-1) to (0,-.6);
				\draw [->, very thick] (0,-.4) to (0,1);
				\node at (1.1,0) {$W$};
				\node at (.25,-.9) {$V$};
			\end{tikzpicture} \right )
			=
			F_\mathcal{D} \left (
			\begin{tikzpicture}[anchorbase,scale=0.8]
				\draw [very thick] (-.1,.5) to [out=180,in=90] (-.8,0) to (-.8,-.4);
				\draw [very thick, postaction={decorate}, decoration={markings, mark=at position 0.4 with {\arrow{>}}}] (-.8,-.6) to [out=270,in=0] (-1,-.8) to [out=180,in=270] (-1.2,-.6) to [out=90,in=180] (-1,-.5) to (.2,-.5) to [out=0,in=270] (.5,.7) to [out=90,in=180] (.7,.9) to [out=0,in=90] (.9,.7) to [out=270,in=0] (.6,.5);
				\draw [very thick] (.1,.5) to (.4, .5);
				\draw [very thick] (0,-1) to (0,-.6);
				\draw [->, very thick] (0,-.4) to (0,1);
				\node at (.86,0) {$W$};
				\node at (.25,-.9) {$V$};
			\end{tikzpicture} \right )
			 \\
			\overset{\RIII}{=} F_\mathcal{D} \left (
			\begin{tikzpicture}[anchorbase,scale=0.8]
				\draw [very thick] (-.8,-.4) to [out=90,in=180] (-.6,0);
				\draw [very thick] (-.4,0) to (-.1,0);
				\draw [very thick, postaction={decorate}, decoration={markings, mark=at position 0.62 with {\arrow{>}}}] (-.8,-.6) to [out=270,in=0] (-1,-.8) to [out=180,in=270] (-1.2,-.6) to [out=90,in=180] (-1,-.5) to (-.7,-.5) to [out=0,in=270] (-.5,.3) to [out=90,in=180](-.4,.4) to (.4,.4) to [out=0,in=90] (.6,.2) to [out=270,in=0] (.4,0) to (.1,0);
				\draw [very thick] (0,-1) to (0,.3);
				\draw [->, very thick] (0,.5) to (0,1);
				\node at (.9,0) {$W$};
				\node at (.25,-.9) {$V$};
			\end{tikzpicture} \right )
			\overset{\RII}{=}
			F_\mathcal{D} \left (
			\begin{tikzpicture}[anchorbase,scale=0.8]
				\draw [very thick, postaction={decorate}, decoration={markings, mark=at position 0.4 with {\arrow{>}}}] (-.1,-.5) to [out=180,in=270] (-.8,0) to [out=90,in=180] (0,.5) to [out=0,in=90] (.8,0) to [out=270,in=0] (.1,-.5);
				\draw [very thick] (0,-1) to (0,.4);
				\draw [->,very thick] (0,.6) to (0,1);
				\node at (1.1,0) {$W$};
				\node at (.25,-.9) {$V$};
			\end{tikzpicture} \right ).
	\end{multline*}
 If $V$ is simple then \cref{rotate_diagram} holds by planar isotopy:
	\begin{equation*}
		\begin{aligned}
		\left \langle	F_\mathcal{D} \left (
			\begin{tikzpicture}[anchorbase,scale=0.8]
				\draw [very thick, postaction={decorate}, decoration={markings, mark=at position 0.4 with {\arrow{>}}}] (-.1,-.5) to [out=180,in=270] (-.8,0) to [out=90,in=180] (0,.5) to [out=0,in=90] (.8,0) to [out=270,in=0] (.1,-.5);
				\draw [very thick] (0,-1) to (0,.4);
				\draw [->, very thick] (0,.6) to (0,1);
				\node at (1.1,0) {$W$};
				\node at (.25,-.9) {$V$};
			\end{tikzpicture} \right ) \right \rangle
			\;&=\;
			\left \langle F_\mathcal{D} \left (
			\begin{tikzpicture}[anchorbase,scale=0.8]
				\draw [very thick, postaction={decorate}, decoration={markings, mark=at position 0.4 with {\arrow{>}}}] (-.1,-.5) to [out=180,in=270] (-.8,0) to [out=90,in=180] (0,.5) to [out=0,in=90] (.8,0) to [out=270,in=0] (.1,-.5);
				\draw [very thick] (-2.4,-1) to (-2.4,.8) to [out=90,in=180] (-1.8,1) to [out=0,in=90] (-1.2,.8) to (-1.2,-.8) to [out=270,in=180] (-.6,-1) to [out=0,in=270] (0,-.8) to (0,.4);
				\draw [->, very thick] (0,.6) to (0,1);
				\node at (1.1,0) {$W$};
				\node at (.25,-.9) {$V$};
			\end{tikzpicture} \right ) \right \rangle
			\\\;&=\;
			\left \langle F_\mathcal{D} \left (
			\begin{tikzpicture}[anchorbase,scale=0.8]
				\draw [very thick, postaction={decorate}, decoration={markings, mark=at position 0.4 with {\arrow{>}}}] (.1,.5) to [out=0,in=90] (.8,0) to [out=270,in=0] (0,-.5) to [out=180,in=270] (-.8,0) to [out=90,in=180] (-.1,.5);
				\draw [->,very thick] (0,-.6) to (0,-.8) to [out=270,in=180] (.6,-1) to [out=0,in=270] (1.2,-.8) to (1.2,1);
				\draw [very thick] (0,-.4) to (0,.8) to [out=90,in=0] (-.6,1) to [out=180,in=90] (-1.2,.8) to (-1.2,-1);
				\node at (.5,.75) {$W$};
				\node at (-.9,-.9) {$V$};
			\end{tikzpicture} \right ) \right \rangle
                \\\;&=\;
                \left \langle F_\mathcal{D} \left (
			\begin{tikzpicture}[anchorbase,scale=0.8]
				\draw [very thick, postaction={decorate}, decoration={markings, mark=at position 0.4 with {\arrow{>}}}] (.1,.5) to [out=0,in=90] (.8,0) to [out=270,in=0] (0,-.5) to [out=180,in=270] (-.8,0) to [out=90,in=180] (-.1,.5);
				\draw [<-,very thick] (0,-1) to (0,-.6);
				\draw [very thick] (0,-.4) to (0,1);
				\node at (1.1,0) {$W$};
				\node at (.4,-.9) {$V$};
			\end{tikzpicture} \right ) \right \rangle
                \left \langle F_\mathcal{D} \left (
			\begin{tikzpicture}[anchorbase,scale=0.8]
				\draw [->,very thick] (0,-.6) to (0,-.8) to [out=270,in=180] (.6,-1) to [out=0,in=270] (1.2,-.8) to (1.2,1);
				\draw [very thick] (0,-.8) to (0,.8) to [out=90,in=0] (-.6,1) to [out=180,in=90] (-1.2,.8) to (-1.2,-1);
				\node at (-.9,-.9) {$V$};
			\end{tikzpicture} \right ) \right \rangle.
		\end{aligned}
	\end{equation*}
The snake relation \eqref{eq:snakeRel} implies that the second scalar in the final line is $1$.
\end{proof}
	
\subsection{Reshetikhin--Turaev invariants and quantum dimension}
	
 Let $K$ be the unknot. Color $K$ by an object $V$ of a $\mathbb{C}$-linear ribbon category $\mathcal{D}$. The scalar $\langle F_\mathcal{D}(K) \rangle$ associated to the map $F_\mathcal{D}(K): \mathbb{C} \rightarrow \mathbb{C}$ is called the \emph{quantum dimension of $V$} and is denoted by $\qdim_\mathcal{D}(V)$. Explicitly, we have $\qdim_\mathcal{D}(V) = \langle \widehat{\ev}_V \circ \coev_V \rangle$.
	
\begin{example}\label{lemmasf}
	Consider again the category $\mathcal{C}$ of weight $\sunrolled$-modules. Let $K$ be the unknot colored by $V_\alpha$, $\alpha \in \mathbb{C}$. Let $\{v_i \mid 0 \leq i \leq r-1\}$ be the weight basis of $V_\alpha$ described in \cref{sec:simpObj} with $\{v_i^\vee \mid 0 \leq i \leq r-1\}$ its dual basis. Then $F_\mathcal{C}(K)$ is the composition
	\[
	1
	\xmapsto{\coev_{V_{\alpha}}} \sum_{i=0}^{r-1} v_i \otimes v_i^\vee
	\xmapsto{\widehat{\ev}_{V_{\alpha}}}
	\sum_{i=0}^{r-1} q^{(\alpha + r - 1 -2i)(1-r)}
	=
	q^{(\alpha + r - 1)(1-r)} \sum_{i=0}^{r-1} q^{-2i + 2ir}.
	\]
	As $q$ is a primitive $2r$\textsuperscript{th} root of unity, we have $\sum_{i=0}^{r-1} q^{-2i + 2ir} = \sum_{i=0}^{r-1} q^{-2i}=0$. Hence, $F_\mathcal{C}(K) = 0$ and $\qdim_\mathcal{C}(V_\alpha)=0$. If instead $K$ is colored by the simple module $S^{lr}_n$, $0 \leq n \leq r-2$ and $l \in \mathbb{Z}$, then
	\[
	\langle F_{\mathcal{C}}(K) \rangle
	=
	\sum_{j=0}^n v_j^{\vee}(K^{1-r} v_j)
	=
	q^{(1-r)(lr+n)} \sum_{j=0}^n q^{-2j}
	=
	(-1)^{n + l + lr} [n+1],
	\]
	whence $\qdim_{\mathcal{C}}(S^{lr}_n) \neq 0$. 
\end{example}

\begin{lemma}\label{cutting}
    Let $\mathcal{D}$ be a $\mathbb{C}$-linear ribbon category, $V \in \mathcal{D}$ a simple object, $L$ a $\mathcal{D}$-colored link and $T$ a $(1,1)$-tangle whose closure is $L$ and whose open strand is colored by $V$. Then
    \begin{equation}\label{formula}
	\langle F_\mathcal{D}(L) \rangle = \qdim_{\mathcal{D}}(V) \langle F_\mathcal{D}(T) \rangle.
    \end{equation} 
\end{lemma}
\begin{proof}
Using isotopy invariance we can draw a diagram of $L$ of the form
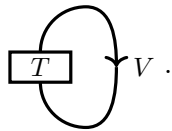
\begin{equation*}
	\begin{tikzpicture}[anchorbase,scale=0.8]
		\draw [very thick] (-.5,-.25) -- (.5,-.25) -- (.5,.25) -- (-.5,.25) -- cycle;
		\draw [<-, very thick] (1.25,0) to [out=90, in=0] (.75,1) to [out=180,in = 90] (0,.25);
		\draw [very thick] (1.25,0) to [out=270, in=00] (.75,-1) to [out=180,in=270] (0,-.25);
		\node at (1.7,0) {\small $V$};
		\node at (0,0) {\small $T$};
	\end{tikzpicture}.
\end{equation*}
Since $V$ is simple, the endomorphism $F_\mathcal{D}(T)$ is a scalar and \cref{formula} follows.
\end{proof}

Thus, whenever a knot is colored by a simple object of vanishing quantum dimension, the Reshetikhin--Turaev invariant is trivial. In particular, in view of \cref{lemmasf}, the Reshetikhin--Turaev invariants of $\mathcal{C}$-colored links with at least one component colored by a simple Verma module are zero.

\subsection{Knot invariants via cutting}
	
\cref{formula} is the starting point of the theory of renormalized quantum invariants of \cite{geer_2009}. The main idea is that even though $\qdim_{\mathcal{D}}(V)$, and hence $F_{\mathcal{D}}(K)$, vanish, $F_{\mathcal{D}}(T)$ need not and may provide an interesting invariant of $K$. In graphical language, to get a non-trivial invariant of a knot $K$ we cut it to obtain a $(1,1)$-tangle $T$ and apply the standard Reshetikhin--Turaev functor to $T$.
	
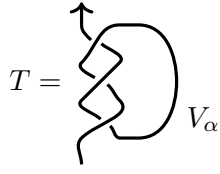
\begin{figure}
    \begin{equation*} \label{trefoil}
	T = 
	\begin{tikzpicture}[anchorbase, scale=0.5]
		\draw [very thick] (-1,0.3) to [out=90, in=225] (-1,1) to [out=45, in=315] (0,2) to [out=135, in=320] (-.32,2.42);
		\draw [very thick] (-.65, 2.6) to [out=150, in=225] (-1,3) to [out=45, in=180] (0,4) to [out=0, in=180] (.5,4) to [out=0, in=90] (1.5, 2.5) to [out=270, in=0] (.5,1) to [out=180, in=0] (0,1) to [out=170, in=320] (-.25,1.32);
		\draw [very thick] (-.5,1.5) to [out=135, in=225] (-1,2) to [out=45, in=225] (0,3) to [out=45, in=320] (-.48,3.5);
		\draw [->,very thick] (-.75,3.6) to [out=135, in=270] (-1,4) to (-1,4.6);
		\node at (2.2,1.5) {$V_\alpha$};.
    \end{tikzpicture}
	\end{equation*}
    \caption{A $(1,1)$-tangle $T$ whose closure is the right-handed trefoil.}
    \label{fig:trefoiltangle}
\end{figure}

\begin{example}
\label{ex:cutTrefoil}
	Let $K$ be the right-handed trefoil knot colored by a Verma module $V_{\alpha} \in \mathcal{C}$. \cref{lemmasf} shows that $\qdim_{\mathcal{C}}(V_{\alpha})=0$. It follows from \cref{cutting} that $F_{\mathcal{C}}(K) = 0$.
	
	Let $T$ be the $(1,1)$-tangle pictured in \cref{fig:trefoiltangle}. The closure of $T$ is $K$. The endomorphism $F_{\mathcal{C}}(T) \in \End_{\mathcal{C}}(V_{\alpha})$ is the composition
	\[
		V_{\alpha} \xrightarrow{\id_{V_{\alpha}} \otimes \coev_{V_{\alpha}}} V_{\alpha} \otimes V_{\alpha} \otimes V_{\alpha}^{\vee} \xrightarrow[]{(c_{V_{\alpha},V_{\alpha}} \otimes \id_{V_{\alpha}^{\vee}})^{\circ 3}} V_{\alpha} \otimes V_{\alpha} \otimes V_{\alpha}^{\vee} \xrightarrow[]{\id_{V_{\alpha}} \otimes \widehat{\ev}_{V_{\alpha}}} V_{\alpha}.
	\]
	Since $V_{\alpha}$ is highest weight, $F_{\mathcal{C}}(T)$ is determined by its value on a highest weight vector $v_0 \in V_{\alpha}$. Using the explicit form of the braiding, we compute
	\[
    \langle F_{\mathcal{C}}(T) \rangle
    =
    q^{\frac{3}{2}(\alpha+r-1)^2+(\alpha+r-1)(1-r)} \sum_{i=0}^{r-1} q^{i(-3\alpha -r+1)}\prod_{j=0}^{i-1} \{i-j-\alpha\}.
    \]
    For example, when $r=2$ and $\alpha=2$, this specializes to $\langle F_{\mathcal{C}}(T) \rangle = - 3 e^{\frac{\pi \sqrt{-1}}{4}} \neq 0$.
\end{example}
	
\begin{proposition} \label{modified_invariant}
    The assignment $K \mapsto \langle F_\mathcal{D}(T) \rangle$, where $K$ is a colored framed knot and $T$ is a $(1,1)$-tangle whose closure is $K$, is a well-defined invariant of colored framed knots.
\end{proposition}

\begin{proof}
This follows from \cref{reshturfunctor} and the standard fact that two connected $(1,1)$-tangles $T$ and $T^{\prime}$ are isotopic if and only if their closures are framed isotopic knots.
\end{proof}
	
\section{Renormalized Reshetikhin--Turaev invariants of $\mathcal{C}$}\label{mqi}
	
We henceforth restrict attention to the ribbon category $\mathcal{C}$ of weight $\sunrolled$-modules and write $F$ for the Reshetikhin--Turaev functor $F_{\mathcal{C}}$.
	
\subsection{Ambidextrous modules} 
	
The idea of constructing a non-zero invariant from a knot by cutting to obtain a $(1,1)$-tangle does not immediately extend to links, as the following example shows.
 
\begin{example}\label{hopfLink}
    Let $\alpha,\beta \in \C$ and $L$ the Hopf link with components colored by $V_{\alpha}$ and $V_{\beta}$. Up to isotopy, there are two choices of how to cut $L$. Cutting the strand colored by $V_\alpha$ gives
    \begin{equation*}
		F\left(
		\begin{tikzpicture}[anchorbase,scale=0.8]
			\draw [very thick, postaction={decorate}, decoration={markings, mark=at position 0.4 with {\arrow{>}}}] (-.1,-.5) to [out=180,in=270] (-.8,0) to [out=90,in=180] (0,.5) to [out=0,in=90] (.8,0) to [out=270,in=0] (.1,-.5);
			\draw [very thick] (0,-1) to (0,.4);
			\draw [->, very thick] (0,.6) to (0,1);
			\node at (1.2,0) {$V_\beta$};
			\node at (.45,-.9) {$V_\alpha$};
		\end{tikzpicture}
		\right) =     \begin{cases} q^{\beta\alpha} \frac{\{\alpha r\}}{\{\alpha\}} \cdot  \id_{V_\alpha} & \mbox{if } \alpha \in \mathbb{C}\setminus r\Z, \\ q^{\beta rz}  \cdot (-1)^{(r + 1)z} r \cdot  \id_{V_\alpha}& \mbox{if } \alpha = rz \in r\Z.  \end{cases}
	\end{equation*} 
	Indeed, the map defined by the above tangle is the composition
	\begin{equation*}
		V_\alpha \xrightarrow{\coev_{V_{\beta}}} V_\alpha \otimes V_\beta \otimes V_\beta^\vee \xrightarrow{c_{V_{\alpha},V_{\beta}}} V_\beta \otimes V_\alpha \otimes V_\beta^\vee \xrightarrow{c_{V_{\beta},V_{\alpha}}} V_\alpha \otimes V_\beta \otimes V_\beta^\vee \xrightarrow{\widehat{\ev}_{V_{\beta}}} V_\alpha .
	\end{equation*}
	As in the proof of \cref{ribbon theorem}, it suffices to compute the image under this map of a highest weight vector $v_0 \in V_\alpha$. Let $\{w_i \mid 0 \leq i \leq r-1\}$ be a weight basis of $V_\beta$ with dual basis $\{w_i^\vee \mid 0 \leq i \leq r-1\}$. Then we have under the above composition
	\begin{multline*}
		v_0 \xmapsto{\id \otimes \coev_{V_\beta}} \sum_{i=0}^{r-1} v_0 \otimes w_i \otimes w_i^\vee
		\xmapsto{c_{V_{\alpha},V_{\beta}} \otimes \id} \sum_{i=0}^{r-1} q^{(\alpha + r - 1)(\beta + r - 1 - 2i)/2} w_i \otimes v_0 \otimes w_i^\vee
		\xmapsto{c_{V_{\beta},V_{\alpha}} \otimes \id}  \\ \sum_{i=0}^{r-1} q^{(\alpha + r - 1)(\beta + r - 1 - 2i)}v_0 \otimes w_i \otimes w_i^\vee + \cdots
		\xmapsto{\id \otimes \widehat{\ev}_{V_\beta}} \sum_{i=0}^{r-1} q^{(\alpha + r - 1)(\beta + r - 1 - 2i)} q^{(\beta + r - 1 - 2i)(1 - r)}v_0,
	\end{multline*}
	where the omitted quantity $\cdots$ is a linear combination of terms of the form $E^j v_0 \otimes w_{i+j} \otimes w_i^{\vee}$, $j >0$, and so is in the kernel of $\id_{V_{\alpha}} \otimes \widehat{\ev}_{V_{\beta}}$. We have
	\[
		\sum_{i=0}^{r-1} q^{(\alpha + r - 1)(\beta + r - 1 - 2i)} q^{(\beta + r - 1 - 2i)(1 - r)}
    =q^{\alpha(\beta + r - 1)}\sum_{i=0}^{r-1}q^{-2\alpha i}.
	\]
	If $\alpha \notin r\Z$, then $q^{-2\alpha} \neq 1$ and the previous line evaluates to 
	\begin{equation*}
		q^{\alpha(\beta + r - 1)} \frac{1 - q^{-2\alpha r}}{1 - q^{-2\alpha}} = q^{\alpha\beta} \frac{q^{\alpha r} - q^{-\alpha r}}{q^{\alpha} - q^{-\alpha}} =q^{\alpha\beta} \frac{\{\alpha r\}}{\{\alpha\}}.
	\end{equation*}
	If instead $\alpha = rz \in r \Z$, then
    \begin{equation*}
		q^{rz(\beta + r - 1)}\sum_{i=0}^{r-1}q^{-2rzi} = q^{rz(\beta + r - 1)} r = (-1)^{rz + z} q^{\beta rz} r.
    \end{equation*}
	
	In particular, taking $r=2$ with $\alpha=0$ and $\beta=2$, we obtain
	\[
	F\left(
	\begin{tikzpicture}[anchorbase,scale=0.8]
    	\draw [very thick, postaction={decorate}, decoration={markings, mark=at position 0.4 with {\arrow{>}}}] (-.1,-.5) to [out=180,in=270] (-.8,0) to [out=90,in=180] (0,.5) to [out=0,in=90] (.8,0) to [out=270,in=0] (.1,-.5);
    	\draw [very thick] (0,-1) to (0,.4);
    	\draw [->, very thick] (0,.6) to (0,1);
    	\node at (1.2,0) {$V_2$};
    	\node at (.35,-.9) {$V_0$};
    \end{tikzpicture}
    \right)
	= 2 \id_{V_0},
	\qquad
	F\left(
	\begin{tikzpicture}[anchorbase,scale=0.8]
		\draw [very thick, postaction={decorate}, decoration={markings, mark=at position 0.4 with {\arrow{>}}}] (-.1,-.5) to [out=180,in=270] (-.8,0) to [out=90,in=180] (0,.5) to [out=0,in=90] (.8,0) to [out=270,in=0] (.1,-.5);
		\draw [very thick] (0,-1) to (0,.4);
		\draw [->, very thick] (0,.6) to (0,1);
		\node at (1.2,0) {$V_0$};
		\node at (.35,-.9) {$V_2$};
	\end{tikzpicture}
	\right)
	= -2 \id_{V_2}.
	\]
	In view of \cref{formula}, we want to attach to the Hopf link colored by $V_0$ and $V_2$ a scalar given by cutting the Hopf link open to a $(1,1)$-tangle. However, we see that the scalar depends non-trivially on which strand we choose to cut.
\end{example}
	
The following notion is the key to resolving the cutting ambiguity illustrated by the previous example.

\begin{definition}[{\cite[\S 3]{geer_2009}}] A module $V \in \mathcal{C}$ is called \emph{ambidextrous} if the equality
	\begin{equation}\label{ambiequation}
		\begin{aligned}
			F\left(
			\begin{tikzpicture}[anchorbase,scale=0.8]
				\draw [very thick] (-.5,-.25) -- (.5,-.25) -- (.5,.25) -- (-.5,.25) -- cycle;
				\draw [<-, very thick] (-1.25,0) to [out=90, in=180] (-.75,1) to [out=0,in = 90] (-.25,.25);
				\draw [very thick] (-1.25,0) to [out=270, in=180] (-.75,-1) to [out=0,in=270] (-.25,-.25);
				\draw [->, very thick] (.25,.25) to [out=90, in=230] (.5,1);
				\draw [very thick] (.25,-.25) to [out=270, in=130] (.5,-1);
				\node at (-1.7,0) {$V$};
				\node at (.8,-.9) {$V$};
				\node at (0,0) {$T$};
			\end{tikzpicture}
			\right)
			&=F\left(
			\begin{tikzpicture}[anchorbase,scale=0.8]
				\draw [very thick] (-.5,-.25) -- (.5,-.25) -- (.5,.25) -- (-.5,.25) -- cycle;
				\draw [->, very thick] (-.25,.25) to [out=90, in=310] (-.5,1);
				\draw [very thick] (-.25,-.25) to [out=270, in=50] (-.5,-1);
				\draw [<-, very thick] (1.25,0) to [out=90, in=0] (.75,1) to [out=180,in = 90] (.25,.25);
				\draw [very thick] (1.25,0) to [out=270, in=00] (.75,-1) to [out=180,in=270] (.25,-.25);
				\node at (1.7,0) {$V$};
				\node at (-.8,-.9) {$V$};
				\node at (0,0) {$T$};
			\end{tikzpicture}
			\right)
		\end{aligned}
    \end{equation}
    holds for all (2,2)-tangles $T$ whose open strands are colored by $V$.
\end{definition}
	
\begin{lemma}\label{ambiqd}
	If $V \in \mathcal{C}$ is simple with non-vanishing quantum dimension, then $V$ is ambidextrous.
\end{lemma}
\begin{proof}
Let $T$ be a $(2,2)$-tangle whose open strands are colored by $V$ and let $L$ be the diagram
    \begin{equation*}
		\begin{tikzpicture}[anchorbase,scale=0.8]
			\draw [very thick] (-.5,-.25) -- (.5,-.25) -- (.5,.25) -- (-.5,.25) -- cycle;
			\draw [<-, very thick] (-1.25,0) to [out=90, in=180] (-.75,1) to [out=0,in = 90] (-.25,.25);
			\draw [very thick] (-1.25,0) to [out=270, in=180] (-.75,-1) to [out=0,in=270] (-.25,-.25);
			\draw [<-, very thick] (1.25,0) to [out=90, in=0] (.75,1) to [out=180,in = 90] (.25,.25);
			\draw [very thick] (1.25,0) to [out=270, in=00] (.75,-1) to [out=180,in=270] (.25,-.25);
			\node at (-1.7,0) {$V$};
			\node at (1.7,0) {$V$};
			\node at (0,0) {$T$};
		\end{tikzpicture}.
	\end{equation*}
	Since $L$ is obtained by taking the right and left partial traces of the $(1,1)$-tangles appearing in \cref{ambiequation}, we find that both sides of this equation are equal to $ \frac{\langle F(L) \rangle}{\qdim_\mathcal{C}(V)}  \id_V$.
\end{proof}

When $V$ has vanishing quantum dimension, we need to investigate further.

\begin{lemma} \label{endocommutative}
	Let $V, W \in \mathcal{C}$ be simple objects such that $V \otimes W$ is semisimple and multiplicity free. Then the algebra $\End_{\mathcal{C}}(V \otimes W)$ is isomorphic to a direct sum of copies of $\mathbb{C}$.
\end{lemma}
\begin{proof}
    Let $U_1, \ldots, U_n$ be pairwise non-isomorphic simples  such that $V \otimes W \simeq U_1 \oplus \cdots \oplus U_n$. Schur's Lemma implies algebra isomorphisms $
	\End_{\mathcal{C}}(V \otimes W)
	\simeq 
	\bigoplus_{i=1}^n \End_{\mathcal{C}}(U_i)
	\simeq
	\mathbb{C}^n$.
\end{proof}
	
\begin{lemma} \label{multiplicityfree}
	Let $\eta \in \mathbb{C} \setminus \frac{1}{2}\Z$. Then $V_\eta \otimes V_\eta \in \mathcal{C}$ is semisimple and multiplicity free.
\end{lemma}
\begin{proof}
Since $2\eta \notin \Z$, the object $V_\eta \otimes V_\eta \in \mathcal{C}_{\overline{2\eta}}$ is semisimple by \cref{Cisgss}.
	Hence, there exist unique integers $m_{2\eta_r -r +2i+1} \in \Z_{\geq 0}$ such that
	\[
	V_\eta \otimes V_\eta \simeq \bigoplus_{i=0}^{r-1} V_{2 \eta -r +2i+1}^{\oplus m_{2\eta_r -r +2i+1}}.
	\]
	
	Consider $\Q[\mathbb{C}]$, the group algebra of $\C$, with basis $\{x^{\lambda}\}_{\lambda \in \mathbb{C}}$. The \emph{character} of $V \in \mathcal{C}$ is $\ch(V) = \sum_{\lambda \in \mathbb{C}} \dim_{\mathbb{C}}(V[\lambda]) x^\lambda$. The explicit description of $V_\alpha$ gives $\ch(V_\alpha) = x^{\alpha}[r]_x$, 	where $[r]_x = \sum_{i=0}^{r-1} x^{r-1-2i}$. We claim that the set
	\[
    \mathcal{S} = \{\ch(V_{2\eta - r + 2i + 1}) \mid 0 \leq i \leq r-1\} \subset \Q[\mathbb{C}]
	\]
	is linearly independent. Suppose that $\sum_{i=0}^{r-1} a_i \ch(V_{2\eta - r + 2i + 1}) = 0$ for some $a_i \in \Q$. Since all powers of $x$ which appear in this equation lie on the same affine real line in $\mathbb{C}$, they are naturally ordered. The largest such power is $x^{2\eta + 2(r-1)}$ with coefficient $a_{r-1}$, resulting from $\ch(V_{2\eta - r + 2(r-1) + 1})$. Hence, $a_{r-1}=0$. Continuing in this way shows that $a_{r-2}=\cdots = a_0=0$ and $\mathcal{S}$ is linearly independent.
		
	The character of $V_\eta \otimes V_\eta$ is $\ch(V_\eta)^2 = x^{2\eta}[r]_x^2$. On the other hand,
	\[
	\ch(V_\eta)^2 = \big( \sum_{i=0}^{r-1} m_{2\eta-r+2i+1}x^{2 \eta-r+2i} \big) [r]_x.
	\]
	Setting each $m_{2\eta-r+1+2i}=1$, the right-hand side of the previous equation becomes
	\[
	\left ( \sum_{i=0}^{r-1} x^{2 \eta-r+2i} \right ) [r]_x
	=
	x^{2\eta}[r]^2_x
	=
	\ch(V_\eta)^2.
	\]
	In view of the linear independence of $\mathcal{S}$, this completes the proof.
\end{proof}
	
\begin{theorem} \label{ambid}
Let $\eta \in \mathbb{C}\setminus\frac{1}{2}\Z$. Then $V_\eta$ is ambidextrous.
\end{theorem}
\begin{proof}
	By \cref{endocommutative,multiplicityfree}, the algebra $\End_{\mathcal{C}}(V_\eta \otimes V_\eta)$ is commutative. Let $T$ be a $(2,2)$-tangle whose open strands are colored by $V_{\eta}$. Then we have the following sequence of equalities, where we implicitly apply $F$ to each tangle and the coupons are colored by $T$:
	\begin{multline*}
			\begin{tikzpicture}[anchorbase,scale=0.8]
				\draw [very thick] (-.5,-.25) -- (.5,-.25) -- (.5,.25) -- (-.5,.25) -- cycle;
				\draw [<-, very thick] (-1.25,0) to [out=90, in=180] (-.75,1) to [out=0,in = 90] (-.25,.25);
				\draw [very thick] (-1.25,0) to [out=270, in=180] (-.75,-1) to [out=0,in=270] (-.25,-.25);
				\draw [->, very thick] (.25,.25) to [out=90, in=230] (.5,1);
				\draw [very thick] (.25,-.25) to [out=270, in=130] (.5,-1);
				\node at (-1.5,-.8) {$V_\eta$};
				\node at (.8,-.8) {$V_\eta$};
			\end{tikzpicture}
			\;\overset{\RII}{=}\;
			\begin{tikzpicture}[anchorbase,scale=0.8]
				\draw [very thick] (-.5,-.25) -- (.5,-.25) -- (.5,.25) -- (-.5,.25) -- cycle;
				\draw [<-, very thick] (-1.25,0) to [out=90, in=180] (-.75,1) to [out=0,in = 90] (-.25,.25);
				\draw [very thick] (-1.25,0) to [out=270, in=180] (-.75,-1) to [out=0,in=270] (.75,-.75);
				\draw [very thick] (-.25,-.25) to [out=270, in=90] (.75,-.75);
				\draw [->, very thick] (.25,.25) to [out=90, in=230] (.5,1);
				\draw [very thick] (.25,-.25) to [out=270, in=90] (.25,-.43);
				\draw [very thick] (.25,-.57) to [out=270, in=90] (.25,-1.03);
				\draw [very thick] (.25,-1.17) to [out=270, in=160] (.5,-1.30);
				\node at (-1.35,-1) {$V_\eta$};
				\node at (.85,.8) {$V_\eta$};
			\end{tikzpicture}
			\;=\;
			\begin{tikzpicture}[anchorbase,scale=0.8]
				\draw [very thick] (-.5,-.25) -- (.5,-.25) -- (.5,.25) -- (-.5,.25) -- cycle;
				\draw [<-, very thick] (-1.25,0) to [out=90, in=180] (-.5,1) to [out=0,in = 90] (.25,.25);
				\draw [very thick] (-1.25,0) to [out=270, in=180] (-.5,-1) to [out=0,in=270] (.25,-.25);
				\draw [very thick] (-.25,.25) to [out=90, in=235] (.05,.65);
				\draw [->, very thick] (.15,.75) to [out=45, in=230] (.5,1);
				\draw [very thick] (-.25,-.25) to [out=270, in=135] (.05,-.65);
				\draw [very thick] (.15,-.75) to [out=315, in=130] (.5,-1);
				\node at (-1.3,-.8) {$V_\eta$};
				\node at (.7,-.8) {$V_\eta$};
			\end{tikzpicture}
			\;\overset{\RI}{=}\; 
			\begin{tikzpicture}[anchorbase,scale=0.8]
				\draw [very thick] (-.5,-.25) -- (.5,-.25) -- (.5,.25) -- (-.5,.25) -- cycle;
				\draw [<-, very thick] (-1.25,0) to [out=90,in=270] (-1.25,1) to [out=90,in=0] (-1.5,1.25) to [out=180,in=90] (-1.75,1) to [out=270,in=180] (-1.33,.75);
				\draw [very thick] (-1.17,.75) to [out=0,in=180] (-.5,1) to [out=0,in=90] (.25,.25);
				\draw [very thick] (-1.25,0) to [out=270,in=90] (-1.25,-1.25) to [out=270,in=0] (-1.5,-1.5) to [out=180,in=270] (-1.75,-1.25) to [out=90,in=180] (-1.33,-1);
				\draw [very thick] (-1.17,-1) to [out=0, in=180] (-.5,-1) to [out=0,in=270] (.25,-.25);
				\draw [very thick] (-.25,.25) to [out=90, in=235] (.05,.65);
				\draw [->, very thick] (.15,.75) to [out=45, in=230] (.5,1);
				\draw [very thick] (-.25,-.25) to [out=270, in=135] (.05,-.65);
				\draw [very thick] (.15,-.75) to [out=315, in=130] (.5,-1);
				\node at (-1.6,-.7) {$V_\eta$};
				\node at (.8,-1) {$V_\eta$};
			\end{tikzpicture}
			\;\overset{\RII}{=}\;
			\\
			\begin{tikzpicture}[anchorbase,scale=0.8]
				\draw [very thick] (-.5,-.25) -- (.5,-.25) -- (.5,.25) -- (-.5,.25) -- cycle;
				\draw [<-, very thick] (-1.25,0) to [out=90,in=180] (-1,.41) to (-.2,.41) to [out=0,in=270] (0,.53) to [out=90,in=0] (-.2,.65) to (-1,.65) to [out=180,in=270] (-1.25,.75) to [out=90,in=270] (-1.25,1) to [out=90,in=0] (-1.5,1.25) to [out=180,in=90] (-1.75,1) to [out=270,in=180] (-1.33,.75);
				\draw [very thick] (-1.17,.75) to [out=0,in=180] (-.5,1) to [out=0,in=90] (.25,.25);
				\draw [very thick] (-1.25,0) to [out=270,in=90] (-1.25,-1.25) to [out=270,in=0] (-1.5,-1.5) to [out=180,in=270] (-1.75,-1.25) to [out=90,in=180] (-1.33,-1);
				\draw [very thick] (-1.17,-1) to [out=0, in=180] (-.5,-1) to [out=0,in=270] (.25,-.25);
				\draw [very thick] (-.25,.25) to [out=90, in=270] (-.25,.35);
				\draw [very thick] (-.25,.45) to [out=90, in=270] (-.25,.6);
				\draw [very thick] (-.25,.7) to [out=90, in=270] (-.25,.9);
				\draw [->, very thick] (-.25,1) to [out=90, in=270] (-.25,1.3);
				\draw [very thick] (-.25,-.25) to [out=270, in=90] (-.25,-.9);
				\draw [very thick] (-.25,-1.03) to [out=270, in=90] (-.25,-1.55);
				\node at (-1.6,-.7) {$V_\eta$};
				\node at (.2,-1.3) {$V_\eta$};
			\end{tikzpicture}
			\;=\;
			\begin{tikzpicture}[anchorbase,scale=0.8]
				\draw [very thick] (-.5,-.25) -- (.5,-.25) -- (.5,.25) -- (-.5,.25) -- cycle;
				\draw [<-, very thick] (-1.25,0) to [out=90,in=180] (-1,.41) to (.5,.41) to [out=0,in=270] (.75,.65) to [out=90,in=0] (.5,.89) to [out=180,in=90] (.25,.65) to (.25,.48);
				\draw [very thick] (.25,.25) to (.25,.35);
				\draw [very thick] (-1.25,0) to [out=270,in=90] (-1.25,-1.25) to [out=270,in=0] (-1.5,-1.5) to [out=180,in=270] (-1.75,-1.25) to [out=90,in=180] (-1.33,-1);
				\draw [very thick] (-1.17,-1) to [out=0, in=180] (-.5,-1) to [out=0,in=270] (.25,-.25);
				\draw [very thick] (-.25,.25) to [out=90, in=270] (-.25,.35);
				\draw [->, very thick] (-.25,.45) to [out=90, in=270] (-.25,1.3);
				\draw [very thick] (-.25,-.25) to [out=270, in=90] (-.25,-.9);
				\draw [very thick] (-.25,-1.03) to [out=270, in=90] (-.25,-1.55);
				\node at (-1.6,-.7) {$V_\eta$};
				\node at (.2,-1.3) {$V_\eta$};
			\end{tikzpicture}
			\;=\;
			\begin{tikzpicture}[anchorbase,scale=0.8]
				\draw [very thick] (-.5,-.25) -- (.5,-.25) -- (.5,.25) -- (-.5,.25) -- cycle;
				\draw [<-, very thick] (-1.25,-.75) to [out=90,in=180] (-.5,-.5) to (.5,-.5) to [out=0,in=270] (.75,0) to [out=90,in=0] (.5,.5) to [out=180,in=90] (.25,.25);
				\draw [very thick] (-1.25,-.75) to [out=270,in=90] (-1.25,-1.25) to [out=270,in=0] (-1.5,-1.5) to [out=180,in=270] (-1.75,-1.25) to [out=90,in=180] (-1.33,-1);
				\draw [very thick] (-1.17,-1) to [out=0, in=180] (-.5,-1) to [out=0,in=270] (.25,-.55);
				\draw [very thick] (.25,-.25) to (.25,-.45);
				\draw [->, very thick] (-.25,.25) to [out=90, in=270] (-.25,1.55);
				\draw [very thick] (-.25,-.25) to [out=270, in=90] (-.25,-.45);
				\draw [very thick] (-.25,-.55) to [out=270, in=90] (-.25,-.9);
				\draw [very thick] (-.25,-1.03) to [out=270, in=90] (-.25,-1.55);
				\node at (-1.7,-.7) {$V_\eta$};
				\node at (-.6,1.1) {$V_\eta$};
			\end{tikzpicture}
			\;=\;
			\begin{tikzpicture}[anchorbase,scale=0.8]
				\draw [very thick] (-.5,-.25) -- (.5,-.25) -- (.5,.25) -- (-.5,.25) -- cycle;
				\draw [very thick, postaction={decorate}, decoration={markings, mark=at position 0.6 with {\arrow{<}}}] (0,-.75) to [out=90,in=180] (.1,-.5) to (.5,-.5) to [out=0,in=270] (.75,0) to [out=90,in=0] (.5,.5) to [out=180,in=90] (.25,.25);
				\draw [very thick] (0,-.75) to [out=270,in=90] (.25,-1) to [out=270,in=0] (.125,-1.15) to [out=180,in=270] (0,-1) to [out=90,in=250] (.06,-.92);
				\draw [very thick] (.17,-.84) to [out=60,in=270] (.25,-.75) to (.25,-.55);
				\draw [very thick] (.25,-.25) to (.25,-.45);
				\draw [->, very thick] (-.25,.25) to [out=90, in=270] (-.25,1.25);
				\draw [very thick] (-.25,-.25) to [out=270, in=90] (-.25,-1.25);
				\node at (-.6,.8) {$V_\eta$};
				\node at (.8,-1) {$V_\eta$};
			\end{tikzpicture}
			\;\overset{\RII}{=}\;
			\begin{tikzpicture}[anchorbase,scale=0.8]
				\draw [very thick] (-.5,-.25) -- (.5,-.25) -- (.5,.25) -- (-.5,.25) -- cycle;
				\draw [<-, very thick] (1.25,0) to [out=90, in=0] (.75,1) to [out=180,in = 90] (.25,.25);
				\draw [very thick] (1.25,0) to [out=270, in=00] (.75,-1) to [out=180,in=270] (.25,-.25);
				\draw [->, very thick] (-.25,.25) to [out=90, in=310] (-.5,1);
				\draw [very thick] (-.25,-.25) to [out=270, in=50] (-.5,-1);
				\node at (-.85,-.8) {$V_\eta$};
				\node at (1.6,-.8) {$V_\eta$};
			\end{tikzpicture}.
		\end{multline*}
	The second equality is implied by the commutativity of $\End_{\mathcal{C}}(V_\eta \otimes V_\eta)$. The fifth and seventh equalities are each a combination of framed Reidemeister moves $\RII$ and $\RIII$. The sixth equality holds by a combination of framed Reidemeister moves that depends on $T$. The other equalities hold by the indicated framed Reidemeister moves.
\end{proof}
	
\subsection{Modified quantum dimensions}

Define a function $S': \mathbb{C} \times \mathbb{C} \rightarrow \mathbb{C}$ by
\begin{equation*}
	S'(\beta, \alpha) = \, \left \langle F \left(
	\begin{tikzpicture}[anchorbase,scale=0.8]
		\draw [very thick, postaction={decorate}, decoration={markings, mark=at position 0.4 with {\arrow{>}}}] (-.1,-.5) to [out=180,in=270] (-.8,0) to [out=90,in=180] (0,.5) to [out=0,in=90] (.8,0) to [out=270,in=0] (.1,-.5);
		\draw [very thick] (0,-1) to (0,.4);
		\draw [->, very thick] (0,.6) to (0,1);
		\node at (1.2,0) {$V_\beta$};
		\node at (.35,-.9) {$V_\alpha$};
	\end{tikzpicture}\right) \right \rangle.
\end{equation*}
	
\begin{proposition}\label{sHopf} 
	The equality
	\begin{equation*}
		S'(\beta, \alpha) = \begin{cases} q^{\beta\alpha} \frac{\{\alpha r\}}{\{\alpha\}} & \mbox{if } \alpha \in \mathbb{C}\setminus r\Z, \\ q^{\beta rz}  \cdot (-1)^{(r + 1)z} r & \mbox{if } \alpha = rz \in r\Z  \end{cases}
	\end{equation*}
 	holds. In particular, $S'(\beta,\alpha)$ is nonzero for all $\alpha \in \mathbb{C} \setminus \Z \cup r\Z$.
\end{proposition}
\begin{proof}
	This was computed in \cref{hopfLink}.
\end{proof}
	
\begin{definition}\label{modifiedqd}
	Let $\eta \in \mathbb{C}$. The \emph{modified quantum dimension with respect to $\eta$} is the function $\mathbf{d}_\eta : \mathbb{C} \setminus \Z \cup r\Z \rightarrow \mathbb{C}$ given by $\mathbf{d}_\eta(\alpha) = \frac{S'(\alpha,\eta)}{S'(\eta,\alpha)}$.
\end{definition}

By \cref{sHopf}, the modified quantum dimension $\mathbf{d}_{\eta}$ is nowhere zero. Modified quantum dimensions associated to different parameters $\eta$ are related as follows.

\begin{proposition}\label{etaetaprime}
	For $\eta, \eta' \in \mathbb{C} \setminus \Z$ and $\alpha \in \mathbb{C} \setminus \Z \cup r\Z$, the following equality holds:
	\[
	\mathbf{d}_\eta(\alpha) = \frac{\sin(\pi \frac{\eta}{r})\sin(\eta' \pi)}{\sin(\eta \pi)\sin(\pi \frac{\eta'}{r})}\mathbf{d}_{\eta'}(\alpha).
	\]
\end{proposition}
\begin{proof}
This follows immediately from \cref{sHopf} and the definition of $\mathbf{d}_{(-)}$.
\end{proof}
	
\begin{theorem}[{\cite[Lemma 2]{geer_2009}}]\label{invariant}
    Let $\eta \in \mathbb{C} $ be such that $V_\eta$ is ambidextrous and $\alpha, \beta \in \mathbb{C} \setminus \Z \cup r\Z$. Then for all $(2,2)$-tangles $T$, the following equality holds:
	\begin{equation*}
		\begin{aligned}
			\mathbf{d}_\eta(\beta)\; \left \langle F \left(
			\begin{tikzpicture}[anchorbase,scale=0.8]
				\draw [very thick] (-.5,-.25) -- (.5,-.25) -- (.5,.25) -- (-.5,.25) -- cycle;
				\draw [<-, very thick] (-1.25,0) to [out=90, in=180] (-.75,1) to [out=0,in = 90] (-.25,.25);
				\draw [very thick] (-1.25,0) to [out=270, in=180] (-.75,-1) to [out=0,in=270] (-.25,-.25);
				\draw [->, very thick] (.25,.25) to [out=90, in=230] (.5,1);
				\draw [very thick] (.25,-.25) to [out=270, in=130] (.5,-1);
				\node at (-1.7,0) {$V_\alpha$};
				\node at (.9,-.9) {$V_\beta$};
				\node at (0,0) {$T$};
			\end{tikzpicture}
			\right) \right \rangle
			&= \mathbf{d}_\eta(\alpha) \; \left \langle F \left(\begin{tikzpicture}[anchorbase,scale=0.8]
				\draw [very thick] (-.5,-.25) -- (.5,-.25) -- (.5,.25) -- (-.5,.25) -- cycle;
				\draw [->, very thick] (-.25,.25) to [out=90, in=310] (-.5,1);
				\draw [very thick] (-.25,-.25) to [out=270, in=50] (-.5,-1);
				\draw [<-, very thick] (1.25,0) to [out=90, in=0] (.75,1) to [out=180,in = 90] (.25,.25);
				\draw [very thick] (1.25,0) to [out=270, in=00] (.75,-1) to [out=180,in=270] (.25,-.25);
				\node at (1.7,0) {$V_\beta$};
				\node at (-.9,-.9) {$V_\alpha$};
				\node at (0,0) {$T$};
			\end{tikzpicture}\right) \right \rangle.
		\end{aligned}
	\end{equation*}
\end{theorem}
\begin{proof}
	Because $V_\eta$ is ambidextrous, there is an equality
	\begin{equation}\label{weird}
		\begin{aligned}
			\left \langle F \left (
			\begin{tikzpicture}[anchorbase,scale=0.8]
				\draw [very thick] (-.5,-.25) -- (.5,-.25) -- (.5,.25) -- (-.5,.25) -- cycle;
				\draw [very thick, postaction={decorate}, decoration={markings, mark=at position 0.7 with {\arrow{<}}}] (-1.25,.3) to [out=90, in=180] (-.75,1) to [out=0,in = 90] (-.25,.25);
				\draw [very thick] (-1.25,.1) to (-1.25,-.25) to [out=270, in=180] (-.75,-1) to [out=0,in=270] (-.25,-.25);
				\draw [very thick, postaction={decorate}, decoration={markings, mark=at position 0.8 with {\arrow{<}}}] (1.15,-.7) to [out=60,in=270] (1.25,0) to [out=90, in=0] (.75,1) to [out=180,in = 90] (.25,.25);
				\draw [very thick] (1.05,-.9) to [out=240, in=0] (.75,-1) to [out=180,in=270] (.25,-.25);
				\draw [very thick, postaction={decorate}, decoration={markings, mark=at position 0.4 with {\arrow{>}}}] (-1.15,-.2) to [out=0, in=270] (-.75,0) to [out=90, in=0] (-1.25,.2) to [out=180, in=90] (-1.75,0) to [out=270, in=180] (-1.35,-.2);
				\draw [very thick, postaction={decorate}, decoration={markings, mark=at position 0.5 with {\arrow{>}}}] (1.4,-1) to [out=90,in=270] (.9,-.6) to (.9, .6) to [out=90,in=45] (.95,.7);
				\draw [very thick] (1.2,.85) to [in=270](1.4, 1);
				\node at (-2.1,0) {$V_\eta$};
				\node at (-1.5,.9) {$V_\alpha$};
				\node at (1.6,0) {$V_\beta$};
				\node at (1.8,-.9) {$V_\eta$};
				\node at (0,0) {$T$};
			\end{tikzpicture} \right ) \right \rangle
			&= \left \langle F \left (
			\begin{tikzpicture}[anchorbase,scale=0.8]
				\draw [very thick] (-.5,-.25) -- (.5,-.25) -- (.5,.25) -- (-.5,.25) -- cycle;
				\draw [very thick, postaction={decorate}, decoration={markings, mark=at position 0.7 with {\arrow{<}}}] (1.25,-.1) to (1.25,.25) to [out=90, in=0] (.75,1) to [out=180,in = 90] (.25,.25);
				\draw [very thick] (1.25,-.3) to [out=270, in=0] (.75,-1) to [out=180,in=270] (.25,-.25);
				\draw [very thick] (-1.15,.7) to [out=240,in=90] (-1.25,0) to [out=270, in=180] (-.75,-1) to [out=0,in=270] (-.25,-.25);
				\draw [very thick, postaction={decorate}, decoration={markings, mark=at position 0.48 with {\arrow{<}}}] (-1.05,.89) to [out=60, in=180] (-.75,1) to [out=0,in=90] (-.25,.25);
				\draw [very thick, postaction={decorate}, decoration={markings, mark=at position 0.65 with {\arrow{<}}}] (1.15,.2) to [out=180, in=90] (.75,0) to [out=270, in=180] (1.25,-.2) to [out=0, in=270] (1.75,0) to [out=90, in=0] (1.35,.2);
				\draw [very thick] (-1.4,-1) to [out=90,in=215] (-1.2,-.85);
				\draw [very thick, postaction={decorate}, decoration={markings, mark=at position 0.5 with {\arrow{>}}}] (-1.02,-.75) to [out=45,in=270] (-.9,-.6) to (-.9, .6) to [out=90,in=270] (-1.4, 1);
				\node at (2.1,0) {$V_\eta$};
				\node at (1.5,.9) {$V_\beta$};
				\node at (-1.6,0) {$V_\alpha$};
				\node at (-1.8,-.9) {$V_\eta$};
				\node at (0,0) {$T$};
			\end{tikzpicture} \right ) \right \rangle.
		\end{aligned}
	\end{equation}
	We expand both sides of this equality. The left-hand side becomes
	\begin{equation*}
		\begin{aligned} &\left \langle F \left (
			\begin{tikzpicture}[anchorbase,scale=0.8]
				\draw [very thick] (-.5,-.25) -- (.5,-.25) -- (.5,.25) -- (-.5,.25) -- cycle;
				\draw [very thick, postaction={decorate}, decoration={markings, mark=at position 0.7 with {\arrow{<}}}] (-1.25,.3) to [out=90, in=180] (-.75,1) to [out=0,in = 90] (-.25,.25);
				\draw [very thick] (-1.25,.1) to (-1.25,-.25) to [out=270, in=180] (-.75,-1) to [out=0,in=270] (-.25,-.25);
				\draw [very thick, postaction={decorate}, decoration={markings, mark=at position 0.8 with {\arrow{<}}}] (1.15,-.7) to [out=60,in=270] (1.25,0) to [out=90, in=0] (.75,1) to [out=180,in = 90] (.25,.25);
				\draw [very thick] (1.05,-.9) to [out=240, in=0] (.75,-1) to [out=180,in=270] (.25,-.25);
				\draw [very thick, postaction={decorate}, decoration={markings, mark=at position 0.4 with {\arrow{>}}}] (-1.15,-.2) to [out=0, in=270] (-.75,0) to [out=90, in=0] (-1.25,.2) to [out=180, in=90] (-1.75,0) to [out=270, in=180] (-1.35,-.2);
				\draw [very thick, postaction={decorate}, decoration={markings, mark=at position 0.5 with {\arrow{>}}}] (1.4,-1) to [out=90,in=270] (.9,-.6) to (.9, .6) to [out=90,in=45] (.95,.7);
				\draw [very thick] (1.2,.85) to [in=270](1.4, 1);
				\node at (-2.1,0) {$V_\eta$};
				\node at (-1.5,.9) {$V_\alpha$};
				\node at (1.6,0) {$V_\beta$};
				\node at (1.8,-.9) {$V_\eta$};
				\node at (0,0) {$T$};
			\end{tikzpicture} \right ) \right \rangle
			= \left \langle F \left (
			\begin{tikzpicture}[anchorbase,scale=0.8]
				\draw [very thick, postaction={decorate}, decoration={markings, mark=at position 0.4 with {\arrow{<}}}] (-.1,-.5) to [out=180,in=270] (-.8,0) to [out=90,in=180] (0,.5) to [out=0,in=90] (.8,0) to [out=270,in=0] (.1,-.5);
				\draw [<-,very thick] (0,-1) to (0,.4);
				\draw [very thick] (0,.6) to (0,1);
				\node at (1.2,0) {$V_\eta$};
				\node at (.5,-.9) {$V_\alpha$};
			\end{tikzpicture} \right ) \right \rangle
			\;\;\left \langle F \left (
			\begin{tikzpicture}[anchorbase,scale=0.8]
				\draw [very thick] (-.5,-.25) -- (.5,-.25) -- (.5,.25) -- (-.5,.25) -- cycle;
				\draw [<-, very thick] (-1.25,0) to [out=90, in=180] (-.75,1) to [out=0,in = 90] (-.25,.25);
				\draw [very thick] (-1.25,0) to [out=270, in=180] (-.75,-1) to [out=0,in=270] (-.25,-.25);
				\draw [very thick, postaction={decorate}, decoration={markings, mark=at position 0.8 with {\arrow{<}}}] (1.15,-.7) to [out=60,in=270] (1.25,0) to [out=90, in=0] (.75,1) to [out=180,in = 90] (.25,.25);
				\draw [very thick] (1.05,-.9) to [out=240, in=0] (.75,-1) to [out=180,in=270] (.25,-.25);
				\draw [very thick, postaction={decorate}, decoration={markings, mark=at position 0.5 with {\arrow{>}}}] (1.4,-1) to [out=90,in=270] (.9,-.6) to (.9, .6) to [out=90,in=45] (.95,.7);
				\draw [very thick] (1.2,.85) to [in=270](1.4, 1);
				\node at (-1.7,0) {$V_\alpha$};
				\node at (1.6,0) {$V_\beta$};
				\node at (1.8,-.9) {$V_\eta$};
				\node at (0,0) {$T$};
			\end{tikzpicture} \right ) \right \rangle
			\\
			&=\left \langle F \left (
			\begin{tikzpicture}[anchorbase,scale=0.8]
				\draw [very thick, postaction={decorate}, decoration={markings, mark=at position 0.4 with {\arrow{<}}}] (.1,.5) to [out=0,in=90] (.8,0) to [out=270,in=0] (0,-.5) to [out=180,in=270] (-.8,0) to [out=90,in=180] (-.1,.5);
				\draw [very thick] (0,-1) to (0,-.6);
				\draw [->, very thick] (0,-.4) to (0,1);
				\node at (1.2,0) {$V_\eta$};
				\node at (.35,-.9) {$V_\alpha$};
			\end{tikzpicture} \right ) \right \rangle
			\;
			\left \langle F \left (
			\begin{tikzpicture}[anchorbase,scale=0.8]
				\draw [very thick] (-.5,-.25) -- (.5,-.25) -- (.5,.25) -- (-.5,.25) -- cycle;
				\draw [<-, very thick] (-1.25,0) to [out=90, in=180] (-.75,1) to [out=0,in = 90] (-.25,.25);
				\draw [very thick] (-1.25,0) to [out=270, in=180] (-.75,-1) to [out=0,in=270] (-.25,-.25);
				\draw [->, very thick] (.25,.25) to [out=90, in=230] (.5,1);
				\draw [very thick] (.25,-.25) to [out=270, in=130] (.5,-1);
				\node at (-1.7,0) {$V_\alpha$};
				\node at (.9,-.9) {$V_\beta$};
				\node at (0,0) {$T$};
			\end{tikzpicture} \right ) \right \rangle
			\;
			\left \langle F \left (
			\begin{tikzpicture}[anchorbase,scale=0.8]
				\draw [very thick, postaction={decorate}, decoration={markings, mark=at position 0.4 with {\arrow{>}}}] (-.1,-.5) to [out=180,in=270] (-.8,0) to [out=90,in=180] (0,.5) to [out=0,in=90] (.8,0) to [out=270,in=0] (.1,-.5);
				\draw [very thick] (0,-1) to (0,.4);
				\draw [->, very thick] (0,.6) to (0,1);
				\node at (1.2,0) {$V_\beta$};
				\node at (.35,-.9) {$V_\eta$};
			\end{tikzpicture} \right ) \right \rangle
			\\
			&=\left \langle F \left (
			\begin{tikzpicture}[anchorbase,scale=0.8]
				\draw [very thick, postaction={decorate}, decoration={markings, mark=at position 0.4 with {\arrow{>}}}] (-.1,-.5) to [out=180,in=270] (-.8,0) to [out=90,in=180] (0,.5) to [out=0,in=90] (.8,0) to [out=270,in=0] (.1,-.5);
				\draw [very thick] (0,-1) to (0,.4);
				\draw [->, very thick] (0,.6) to (0,1);
				\node at (1.2,0) {$V_\eta$};
				\node at (.35,-.9) {$V_\alpha$};
			\end{tikzpicture} \right ) \right \rangle
			\;
			\left \langle F \left (
			\begin{tikzpicture}[anchorbase,scale=0.8]
				\draw [very thick] (-.5,-.25) -- (.5,-.25) -- (.5,.25) -- (-.5,.25) -- cycle;
				\draw [<-, very thick] (-1.25,0) to [out=90, in=180] (-.75,1) to [out=0,in = 90] (-.25,.25);
				\draw [very thick] (-1.25,0) to [out=270, in=180] (-.75,-1) to [out=0,in=270] (-.25,-.25);
				\draw [->, very thick] (.25,.25) to [out=90, in=230] (.5,1);
				\draw [very thick] (.25,-.25) to [out=270, in=130] (.5,-1);
				\node at (-1.7,0) {$V_\alpha$};
				\node at (.9,-.9) {$V_\beta$};
				\node at (0,0) {$T$};
			\end{tikzpicture} \right ) \right \rangle
			\;
			\left \langle F \left (
			\begin{tikzpicture}[anchorbase,scale=0.8]
				\draw [very thick, postaction={decorate}, decoration={markings, mark=at position 0.4 with {\arrow{>}}}] (-.1,-.5) to [out=180,in=270] (-.8,0) to [out=90,in=180] (0,.5) to [out=0,in=90] (.8,0) to [out=270,in=0] (.1,-.5);
				\draw [very thick] (0,-1) to (0,.4);
				\draw [->, very thick] (0,.6) to (0,1);
				\node at (1.2,0) {$V_\beta$};
				\node at (.35,-.9) {$V_\eta$};
			\end{tikzpicture} \right ) \right \rangle
			\\
			&= S'(\eta,\alpha)\;
			\left \langle F \left (
			\begin{tikzpicture}[anchorbase,scale=0.8]
				\draw [very thick] (-.5,-.25) -- (.5,-.25) -- (.5,.25) -- (-.5,.25) -- cycle;
				\draw [<-, very thick] (-1.25,0) to [out=90, in=180] (-.75,1) to [out=0,in = 90] (-.25,.25);
				\draw [very thick] (-1.25,0) to [out=270, in=180] (-.75,-1) to [out=0,in=270] (-.25,-.25);
				\draw [->, very thick] (.25,.25) to [out=90, in=230] (.5,1);
				\draw [very thick] (.25,-.25) to [out=270, in=130] (.5,-1);
				\node at (-1.7,0) {$V_\alpha$};
				\node at (.9,-.9) {$V_\beta$};
				\node at (0,0) {$T$};
			\end{tikzpicture} \right ) \right \rangle\; S'(\beta,\eta),
		\end{aligned}
	\end{equation*}
	where we have repeatedly applied \cref{rotate}. The right-hand side becomes
	\begin{equation*}
		\begin{aligned}
		    &\left \langle F \left (
			\begin{tikzpicture}[anchorbase,scale=0.8]
				\draw [very thick] (-.5,-.25) -- (.5,-.25) -- (.5,.25) -- (-.5,.25) -- cycle;
				\draw [very thick, postaction={decorate}, decoration={markings, mark=at position 0.7 with {\arrow{<}}}] (1.25,-.1) to (1.25,.25) to [out=90, in=0] (.75,1) to [out=180,in = 90] (.25,.25);
				\draw [very thick] (1.25,-.3) to [out=270, in=0] (.75,-1) to [out=180,in=270] (.25,-.25);
				\draw [very thick] (-1.15,.7) to [out=240,in=90] (-1.25,0) to [out=270, in=180] (-.75,-1) to [out=0,in=270] (-.25,-.25);
				\draw [very thick, postaction={decorate}, decoration={markings, mark=at position 0.48 with {\arrow{<}}}] (-1.05,.89) to [out=60, in=180] (-.75,1) to [out=0,in=90] (-.25,.25);
				\draw [very thick, postaction={decorate}, decoration={markings, mark=at position 0.65 with {\arrow{<}}}] (1.15,.2) to [out=180, in=90] (.75,0) to [out=270, in=180] (1.25,-.2) to [out=0, in=270] (1.75,0) to [out=90, in=0] (1.35,.2);
				\draw [very thick] (-1.4,-1) to [out=90,in=215] (-1.2,-.85);
				\draw [very thick, postaction={decorate}, decoration={markings, mark=at position 0.5 with {\arrow{>}}}] (-1.02,-.75) to [out=45,in=270] (-.9,-.6) to (-.9, .6) to [out=90,in=270] (-1.4, 1);
				\node at (2.1,0) {$V_\eta$};
				\node at (1.5,.9) {$V_\beta$};
				\node at (-1.6,0) {$V_\alpha$};
				\node at (-1.8,-.9) {$V_\eta$};
				\node at (0,0) {$T$};
			\end{tikzpicture}
			\right ) \right \rangle
			= \left \langle F \left (
			\begin{tikzpicture}[anchorbase,scale=0.8]
				\draw [very thick] (-.5,-.25) -- (.5,-.25) -- (.5,.25) -- (-.5,.25) -- cycle;
				\draw [<-, very thick] (1.25,0) to [out=90, in=0] (.75,1) to [out=180,in = 90] (.25,.25);
				\draw [very thick] (1.25,0) to [out=270, in=00] (.75,-1) to [out=180,in=270] (.25,-.25);
				\draw [very thick, postaction={decorate}, decoration={markings, mark=at position 0.48 with {\arrow{<}}}] (-1.05,.89) to [out=60, in=180] (-.75,1) to [out=0,in=90] (-.25,.25);
				\draw [very thick] (-1.15,.7) to [out=240,in=90] (-1.25,0) to [out=270, in=180] (-.75,-1) to [out=0,in=270] (-.25,-.25);
				\draw [very thick] (-1.4,-1) to [out=90,in=215] (-1.2,-.85);
				\draw [very thick, postaction={decorate}, decoration={markings, mark=at position 0.5 with {\arrow{>}}}] (-1.02,-.75) to [out=45,in=270] (-.9,-.6) to (-.9, .6) to [out=90,in=270] (-1.4, 1);
				\node at (1.7,0) {$V_\beta$};
				\node at (-1.6,0) {$V_\alpha$};
				\node at (-1.8,-.9) {$V_\eta$};
				\node at (0,0) {$T$};
			\end{tikzpicture} \right ) \right \rangle
			\;\;
			\left \langle F \left (
			\begin{tikzpicture}[anchorbase,scale=0.8]
				\draw [very thick, postaction={decorate}, decoration={markings, mark=at position 0.4 with {\arrow{>}}}] (.1,.5) to [out=0,in=90] (.8,0) to [out=270,in=0] (0,-.5) to [out=180,in=270] (-.8,0) to [out=90,in=180] (-.1,.5);
				\draw [<-,very thick] (0,-1) to (0,-.6);
				\draw [very thick] (0,-.4) to (0,1);
				\node at (1.2,0) {$V_\eta$};
				\node at (.5,-.9) {$V_\beta$};
			\end{tikzpicture} \right ) \right \rangle
			\\
			&= \left \langle F \left (
			\begin{tikzpicture}[anchorbase,scale=0.8]
				\draw [very thick, postaction={decorate}, decoration={markings, mark=at position 0.6 with {\arrow{<}}}] (.1,.5) to [out=0,in=90] (.8,0) to [out=270,in=0] (0,-.5) to [out=180,in=270] (-.8,0) to [out=90,in=180] (-.1,.5);
				\draw [very thick] (0,-1) to (0,-.6);
				\draw [->, very thick] (0,-.4) to (0,1);
				\node at (1.2,0) {$V_\alpha$};
				\node at (.35,-.9) {$V_\eta$};
			\end{tikzpicture} \right ) \right \rangle
			\; \left \langle F \left (
			\begin{tikzpicture}[anchorbase,scale=0.8]
				\draw [very thick] (-.5,-.25) -- (.5,-.25) -- (.5,.25) -- (-.5,.25) -- cycle;
				\draw [->, very thick] (-.25,.25) to [out=90, in=310] (-.5,1);
				\draw [very thick] (-.25,-.25) to [out=270, in=50] (-.5,-1);
				\draw [<-, very thick] (1.25,0) to [out=90, in=0] (.75,1) to [out=180,in = 90] (.25,.25);
				\draw [very thick] (1.25,0) to [out=270, in=00] (.75,-1) to [out=180,in=270] (.25,-.25);
				\node at (1.7,0) {$V_\beta$};
				\node at (-.9,-.9) {$V_\alpha$};
				\node at (0,0) {$T$};
			\end{tikzpicture} \right ) \right \rangle
			\; \left \langle F \left (
			\begin{tikzpicture}[anchorbase,scale=0.8]
				\draw [very thick, postaction={decorate}, decoration={markings, mark=at position 0.4 with {\arrow{>}}}] (-.1,-.5) to [out=180,in=270] (-.8,0) to [out=90,in=180] (0,.5) to [out=0,in=90] (.8,0) to [out=270,in=0] (.1,-.5);
				\draw [very thick] (0,-1) to (0,.4);
				\draw [->, very thick] (0,.6) to (0,1);
				\node at (1.2,0) {$V_\eta$};
				\node at (.35,-.9) {$V_\beta$};
			\end{tikzpicture} \right ) \right \rangle
			\\
			&= \left \langle F \left (
			\begin{tikzpicture}[anchorbase,scale=0.8]
				\draw [very thick, postaction={decorate}, decoration={markings, mark=at position 0.4 with {\arrow{>}}}] (-.1,-.5) to [out=180,in=270] (-.8,0) to [out=90,in=180] (0,.5) to [out=0,in=90] (.8,0) to [out=270,in=0] (.1,-.5);
				\draw [very thick] (0,-1) to (0,.4);
				\draw [->, very thick] (0,.6) to (0,1);
				\node at (1.2,0) {$V_\alpha$};
				\node at (.35,-.9) {$V_\eta$};
			\end{tikzpicture} \right ) \right \rangle
			\; \left \langle F \left (
			\begin{tikzpicture}[anchorbase,scale=0.8]
				\draw [very thick] (-.5,-.25) -- (.5,-.25) -- (.5,.25) -- (-.5,.25) -- cycle;
				\draw [->, very thick] (-.25,.25) to [out=90, in=310] (-.5,1);
				\draw [very thick] (-.25,-.25) to [out=270, in=50] (-.5,-1);
				\draw [<-, very thick] (1.25,0) to [out=90, in=0] (.75,1) to [out=180,in = 90] (.25,.25);
				\draw [very thick] (1.25,0) to [out=270, in=00] (.75,-1) to [out=180,in=270] (.25,-.25);
				\node at (1.7,0) {$V_\beta$};
				\node at (-.9,-.9) {$V_\alpha$};
				\node at (0,0) {$T$};
			\end{tikzpicture} \right ) \right \rangle
			\; \left \langle F \left (
			\begin{tikzpicture}[anchorbase,scale=0.8]
				\draw [very thick, postaction={decorate}, decoration={markings, mark=at position 0.4 with {\arrow{>}}}] (-.1,-.5) to [out=180,in=270] (-.8,0) to [out=90,in=180] (0,.5) to [out=0,in=90] (.8,0) to [out=270,in=0] (.1,-.5);
				\draw [very thick] (0,-1) to (0,.4);
				\draw [->, very thick] (0,.6) to (0,1);
				\node at (1.2,0) {$V_\eta$};
				\node at (.35,-.9) {$V_\beta$};
			\end{tikzpicture} \right ) \right \rangle
			\\
			&=S'(\alpha,\eta)\; \left \langle F \left (
			\begin{tikzpicture}[anchorbase,scale=0.8]
				\draw [very thick] (-.5,-.25) -- (.5,-.25) -- (.5,.25) -- (-.5,.25) -- cycle;
				\draw [->, very thick] (-.25,.25) to [out=90, in=310] (-.5,1);
				\draw [very thick] (-.25,-.25) to [out=270, in=50] (-.5,-1);
				\draw [<-, very thick] (1.25,0) to [out=90, in=0] (.75,1) to [out=180,in = 90] (.25,.25);
				\draw [very thick] (1.25,0) to [out=270, in=00] (.75,-1) to [out=180,in=270] (.25,-.25);
				\node at (1.7,0) {$V_\beta$};
				\node at (-.9,-.9) {$V_\alpha$};
				\node at (0,0) {$T$};
			\end{tikzpicture} \right ) \right \rangle\; S'(\eta,\beta).
		\end{aligned}
	\end{equation*}
	By \cref{sHopf}, $S'(\eta,\alpha)$ and $S'(\eta,\beta)$ are non-zero. We can therefore divide both sides of \cref{weird} by $S'(\eta,\alpha)S'(\eta,\beta)$ to complete the proof.
\end{proof}
	
\begin{corollary}\label{manyAmbi}
	For each $\alpha \in \mathbb{C}\setminus \Z \cup r\Z$, the module $V_\alpha$ is ambidextrous. In particular, any simple module in $\mathcal{C}$ is ambidextrous.
\end{corollary}
\begin{proof}
    Let $\eta \in \mathbb{C} \setminus \frac{1}{2} \mathbb{Z}$. By \cref{ambid}, the module $V_{\eta}$ is ambidextrous. By \cref{sHopf}, the scalar $\mathbf{d}_{\eta}(\alpha)$ is non-zero. The first statement now follows from taking $\alpha =\beta$ in \cref{invariant}. Using the classification of simple objects of $\mathcal{C}$ given in \cref{simplemodules}, the second claim follows from the above and \cref{ambiqd}.
\end{proof}
	
\subsection{Renormalized Reshetikhin--Turaev invariants of links}

Denote by $\mathfrak{L}$ the set of all framed colored links for which at least one of its colors is of the form $V_\alpha$ for some  $\alpha \in \mathbb{C} \setminus \Z \cup r\Z$. We view $\mathfrak{L}$ as a subset of morphisms of $\mathbf{Rib}_\mathcal{C}$.

\begin{theorem}[{\cite[Theorem 3]{geer_2009}}]\label{invaraiant}
    Let $\eta \in \mathbb{C} \setminus \Z \cup r \Z$. Then the map $F_\eta': \mathfrak{L} \rightarrow \mathbb{C}$ given by
	\begin{equation*}
		F_\eta'(L) = \mathbf{d}_\eta(\alpha) \langle F(T) \rangle,
	\end{equation*}
	where $T$ is a $(1,1)$-tangle whose closure is $L$ and whose open strand is colored by $V_{\alpha}$ for some $\alpha \in \mathbb{C} \setminus \mathbb{Z} \cup r \mathbb{Z}$, is a well-defined isotopy invariant of links in $\mathfrak{L}$.
\end{theorem}
\begin{proof}
    Well-definedness of $F_\eta'$ is the statement that $F_\eta'(L)$ is independent of the choice of $(1,1)$-tangle $T$ used in its definition. If $T$ and $T'$ are $(1,1)$-tangles constructed from $L$ by cutting along strands colored by $V_\alpha$ and $V_\beta$, respectively, then, in view of \cref{manyAmbi}, \cref{invariant} shows that $\mathbf{d}_\eta(\alpha) \langle F(T) \rangle = \mathbf{d}_\eta(\beta) \langle F(T^{\prime}) \rangle$. Isotopy invariance of $F_\eta'$ follows from \cref{reshturfunctor}.
\end{proof}

Comparing the definition of $ F_\eta'(L)$ with \cref{formula} shows that $\mathbf{d}_{\eta}(\alpha)$ plays the role of $\qdim_{\mathcal{C}}(V_{\alpha})$ in the standard theory. This justifies the term \emph{modified quantum dimension} for the function $\mathbf{d}_{\eta}$.

\subsection{Basic properties and examples}
\label{sec:examples}

We begin by discussing the dependence of the renormalized invariants of \cref{invaraiant} on the parameter $\eta$. By \cref{etaetaprime}, the modified quantum dimensions functions associated to two such parameters $\eta$ and $\eta^{\prime}$ differ by an explicit global scalar. In fact, more is true.

\begin{lemma}\label{multiples}
	Let $D: \mathbb{C}\setminus \Z \cup r \Z \rightarrow \mathbb{C}$ be a function such that the assignment $L \mapsto D(\alpha) \langle F(T) \rangle$, where $T$ is any $(1,1)$-tangle whose closure is isotopic to $L$ and whose open strand is colored by $V_{\alpha}$ for some $\alpha \in \mathbb{C} \setminus \mathbb{Z} \cup r \mathbb{Z}$, is a well-defined invariant of colored framed links in $\mathfrak{L}$. Then, for any $\eta \in \C \setminus \Z$, there exists a scalar $d \in \C$ such that $D = d \cdot \mathbf{d}_\eta$.
\end{lemma}

\begin{proof}
	Let $L$ be a Hopf link with strands colored by $V_\alpha$ and $V_{\eta}$ for some $\alpha \in \mathbb{C} \setminus \Z \cup r\Z$ and $\eta \in \C \setminus \Z$. The assumption of well-definedness of the invariant implies that $D(\alpha) S'(\eta, \alpha) = D(\eta) S'(\alpha, \eta)$, whence 
	$D(\alpha) = D(\eta) \mathbf{d}_\eta(\alpha)$. Taking $d = D(\eta)$ proves the lemma. 
\end{proof}

It follows from \cref{multiples} that renormalized invariants arising from a modified quantum dimension $\mathbf{d}_\eta$ are effectively the only invariants that incorporate the cutting process described above. When $r=2$, \cref{multiples} can be strengthened. By \cref{etaetaprime}, any function $D$ that gives an invariant incorporating the cutting procedure is equal to $\pm \mathbf{d}_\eta$ for some $\eta \in \mathbb{C}$, that is, the scalar $d$ in \cref{multiples} can be taken to be $\pm 1$.

\begin{remark}
    By \cref{lemmasf}, the simple modules $S^{lr}_n$ have non-zero quantum dimension and so, by \cref{ambiqd}, are ambidextrous. In view of this, there is a modification of \cref{invaraiant} in which the standard Reshetikhin--Turaev invariant is renormalized with respect to $S^{lr}_n$ instead of $V_{\eta}$. The resulting invariant $F'_{n,lr}$ is defined on $\mathfrak{L}$ as well as those knots $\tilde{\mathfrak{L}}$ colored by at least one simple module of the form $S_{m}^{k r}$. A calculation as in \cref{hopfLink} shows that the $S^{lr}_n$-renormalized quantum dimension $\mathbf{d}_n^{lr}$ vanishes on  $V_\alpha$, $\alpha \in \mathbb{C} \setminus \mathbb{Z} \cup r \mathbb{Z}$. In particular, $F'_{n,lr}$ is zero on $\mathfrak{L}$. On the other hand, a direct calculation gives
    \[
    \qdim_\mathcal{C}(S^{kr}_m) = \qdim_\mathcal{C}(S^{lr}_n) \cdot \mathbf{d}_n^{lr}(S^{kr}_m).
    \]    
    It follows that the restriction of $F'_{n,lr}$ to $\tilde{\mathfrak{L}}$ is equal to $\qdim_\mathcal{C}(S^{lr}_n)^{-1}$ times the standard Reshetikhin--Turaev invariant $F$. From this perspective, the renormalized theory recovers the standard theory.
\end{remark}

Next, we describe the behavior of renormalized invariants under connect sum.

\begin{proposition}
    Let $L, L' \in \mathcal{L}$, each of which have at least one strand colored by $V_\alpha$ for some $\alpha \in \C \setminus \Z \cup r \Z$. Then the $\eta$-renormalized Reshetikhin--Turaev invariant of the connect sum $L\#L'$ along strands colored by $V_\alpha$ satisfies
    $$\mathbf{d}_\eta(\alpha) F^{\prime}_{\eta}(L\#L^{\prime})  =  F^{\prime}_{\eta}(L) \cdot  F^{\prime}_{\eta}(L^{\prime}) .$$
\end{proposition}

\begin{proof}
Consider a knot diagram for $L \# L^{\prime}$ of the form
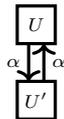
\[
\begin{tikzpicture}[anchorbase,scale=1.0]
			\draw [very thick](-.25,.25) -- (-.25,.75) -- (.25,.75) -- (.25,.25) -- cycle;
			\draw [very thick](-.25,-.25) -- (-.25,-.75) -- (.25,-.75) -- (.25,-.25) -- cycle;
			\draw[very thick,->](-.1,.25) -- (-.1,-.25);
			\draw[very thick,<-](.1,.25) -- (.1,-.25);
	\node at (0,.5) {\tiny $U$};
    \node at (0,-.5) {\tiny $U'$};
    \node at (-.3,0) {\tiny $\alpha$};
    \node at (.3,0) {\tiny $\alpha$};
	\end{tikzpicture}
	\]
	where the obvious closures of the $(2,0)$-tangle $U$ and $(0,2)$-tangle $U^{\prime}$ are knot diagrams for $L$ and $L^{\prime}$, respectively, and we have written $\alpha$ for the color $V_{\alpha}$. Cutting this diagram along one of the connecting strands gives
	\[
	F^{\prime}_{\eta}(L \# L^{\prime})
	=
\mathbf{d}_\eta(\alpha)\left \langle F \left (
        \begin{tikzpicture}[anchorbase,scale=0.8]
			\draw [very thick](-.25,.25) -- (-.25,.75) -- (.25,.75) -- (.25,.25) -- cycle;
			\draw [very thick](-.25,-.25) -- (-.25,-.75) -- (.25,-.75) -- (.25,-.25) -- cycle;
			\draw[very thick,->](0,.75) -- (0,1.2);
			\draw[very thick,<-](0,.25) -- (0,-.25);
			\draw[very thick,<-](0,-.75) -- (0,-1.2);
	\node at (0,.5) {\tiny $U$};
    \node at (0,-.5) {\tiny $U'$};
    \node at (-.33,1) {\tiny $\alpha$};
    \node at (-.35,0) {\tiny $\alpha$};
    \node at (-.35,-1) {\tiny $\alpha$};
	\end{tikzpicture} \right ) \right \rangle
    \]
    where, by a slight abuse of notation, we now denote by $U$ the $(1,1)$-tangle obtained from $U$ by pulling one open strand to the top, and similarly for $U^{\prime}$. Since $V_\alpha$ is simple, the right hand side of this equation is equal to \[
    \mathbf{d}_\eta(\alpha)
    \left \langle F \left (
        \begin{tikzpicture}[anchorbase,scale=0.8]
			\draw [very thick](-.25,.25) -- (-.25,.75) -- (.25,.75) -- (.25,.25) -- cycle;
			\draw[very thick,->](0,.75) -- (0,1.2);
			\draw[very thick,<-](0,.25) -- (0,-.25);
	\node at (0,.5) {\tiny $U$};
    \node at (-.35,1) {\tiny $\alpha$};
    \node at (-.35,0) {\tiny $\alpha$};
	\end{tikzpicture} \right ) \right \rangle
\left \langle F \left (
        \begin{tikzpicture}[anchorbase,scale=0.8]
			\draw [very thick](-.25,.25) -- (-.25,.75) -- (.25,.75) -- (.25,.25) -- cycle;
			\draw[very thick,->](0,.75) -- (0,1.2);
			\draw[very thick,<-](0,.25) -- (0,-.25);
	\node at (0,.5) {\tiny $U^{\prime}$};
    \node at (-.35,1) {\tiny $\alpha$};
    \node at (-.35,0) {\tiny $\alpha$};
	\end{tikzpicture} \right ) \right \rangle
    =
    \mathbf{d}_\eta(\alpha)^{-1} F^{\prime}_{\eta}(L) F^{\prime}_{\eta}(L^{\prime}).
    \]
    This gives the desired expression for $F^{\prime}_{\eta}(L \# L^{\prime})$.
\end{proof}
	
\begin{example}
	Let $r=2$. Set $V=V_{\alpha}$ for some $\alpha \in \C$. Consider $V \otimes V$ with basis $\{v_0 \otimes v_0, v_0 \otimes v_1, v_1 \otimes v_0, v_1 \otimes v_1\}$. A direct computation gives
    \[
    c_{V,V} = \left( 
    \begin{smallmatrix}
        q^{\frac{(\alpha+1)^2}{2}} & 0 & 0 & 0 \\
        0 & 0 & q^{\frac{(\alpha+1)(\alpha-1)}{2}} & 0 \\
        0 & q^{\frac{(\alpha+1)(\alpha-1)}{2}} & q^{\frac{(\alpha+1)(\alpha-1)}{2}}\{1-\alpha\} & 0 \\
        0 & 0 & 0 & q^{\frac{(\alpha-1)^2}{2}}
    \end{smallmatrix}
    \right).
    \]
    For example, the $(3,2)$ and $(3,3)$ entries of $c_{V,V}$ are the coefficients of $c_{V,V}(v_1 \otimes v_0)$:
    \begin{eqnarray*}
        c_{V,V}(v_1 \otimes v_0)
        &=&
        \tau (q^{H\otimes H/2} (1 +\frac{\{1\}^2}{\{1\}}E \otimes F) (v_1 \otimes v_0) \\
        &=&
        \tau (q^{H\otimes H/2} (v_1 \otimes v_0 +\frac{\{1\}^2}{\{1\}}Ev_1 \otimes Fv_0) \\
        &=&
        q^{\frac{(\alpha-1)(\alpha+1)}{2}}v_0 \otimes v_1 +q^{\frac{(\alpha-1)(\alpha+1)}{2}}\{1-\alpha\} v_1 \otimes v_0.
    \end{eqnarray*}
    Using the explicit formula for $c_{V,V}$, we verify the equality
    \begin{equation} \label{alexSkein}
    q^{-\frac{(\alpha+1)(\alpha-1)}{2}}c_{V,V}- q^{\frac{(\alpha+1)(\alpha-1)}{2}}c_{V,V}^{-1} = \{\alpha+1\} \id_{V \otimes V}.
    \end{equation}
    Recall from the proof of \cref{ribbon theorem} that $\theta_V = q^{-\frac{(\alpha+1)(\alpha-1)}{2}} \id_V$. Define $\mathcal{F}^{\prime}(L)=q^{\frac{(\alpha+1)(\alpha-1)}{2} \writhe(L)}F^{\prime}(L)$, where $\writhe(L)$ is the writhe of $L$. Then $\mathcal{F}^{\prime}$ is an invariant of oriented links colored by $V_{\alpha}$. This is an instance of the deframing procedure, explained, for example, in \cite[\S 3.3]{jackson2019}. The relation \eqref{alexSkein} implies that $\mathcal{F}^{\prime}$ satisfies the Alexander skein relation with $t=q^{\alpha+1}$.
\end{example}

We end this section with some calculations of renormalized invariants for knots with few crossings.

\begin{example}\label{ex:renormalizedTrefoil}
Using the result of \cref{ex:cutTrefoil}, the renormalized invariant of the right-handed trefoil $K=3_1$ colored by a simple module $V_{\alpha}$ is 
\[
F^{\prime}_{\eta}(K)
=
\frac{\{\eta r\} \{\alpha\}}{\{\eta\}\{\alpha r\}} q^{\frac{3}{2}(\alpha+r-1)^2 + (\alpha+r-1)(1-r)} \sum_{i=0}^{r-1} q^{i(-3\alpha - r +i)} \prod_{j=0}^{i-1} \{i-j-\alpha\}.\qedhere
\]
\end{example}

\begin{example}
    Proceeding analogously to \cref{ex:renormalizedTrefoil}, the renormalized invariant of the figure eight knot $K=4_1$ colored by a simple module $V_\alpha$ is
    \begin{multline*}
    F^{\prime}_{\eta}(K)
    =
    \frac{\{\eta r\} \{\alpha\}}{\{\eta\}\{\alpha r\}} 
        \sum_{i=0}^{r-1} \sum_{j=0}^{i} \sum_{k=0}^{r-i} q^{(-\alpha + r - 1 -2i)(\alpha - r + 1)/2} q^{-(-\alpha + r - 1 -2(i-j))(i + r-1-j)} \cdot \\
         q^{-(-\alpha + r - 1 - 2(i-j +k))(\alpha + r - 1 -2(i+k))/2} q^{j(j-1)/2} q^{k(k-1)/2} q^{(i + k - r + 1)(i + k - r)/2} \cdot \\
        \prod_{x = r-1}^{r-1-j}  \frac{\{x\} \{x- \alpha\}}{\{j\}!} \prod_{y = i-j}^{i-j+k}  \frac{\{y\} \{y + \alpha\}}{\{k\}!} ) \prod_{z = i+k}^{r-1} \frac{\{z\} \{z - \alpha\}}{\{l\}!}  . \qedhere
    \end{multline*}
\end{example}

\section{Further reading}\label{sec:furtherReading}

In this final section, we give a brief sample of recent developments in renormalized Reshetikhin--Turaev theory.

The construction of renormalized Reshetikhin--Turaev invariants of links, given in \cref{invariant} for the category of weight modules over $\sunrolled$, applies more generally to certain non-semisimple ribbon categories \cite{geer_2009}. The key new notions in this more general context are \emph{modified traces} (and so modified quantum dimensions) and ambidextrous objects \cite{GKP1,GKP2,GKP3,geer_2017,fontalvo2018,gainutdinov2020,beliakova2021}. Renormalized Reshetikhin--Turaev invariants of links have been studied for categories of weight modules over unrolled quantum groups of complex simple Lie algebras \cite{geer2013} and Lie superalgebras \cite{geer2007b,geer2010}. The renormalized invariants are motivated by and recover previous \emph{ad hoc} renormalized invariants \cite{kauffman1991,akutsu1992,viro2006,geer2010}.

The extension of the renormalized theory from links to $3$-manifolds was achieved in \cite{costantino2014quantum} for the category of weight modules over $\sunrolled$, where properties \ref{ite:nonss} and \ref{ite:infSimp} are serious obstructions. In general, the additional data and constraints on the input ribbon category required to obtain a $3$-manifold invariant are termed \emph{non-degenerate relative pre-modularity}. The resulting $3$-manifold invariants have novel properties, including the ability to distinguish homotopy classes of lens spaces and connections with the Volume Conjecture, and are related to earlier non-semisimple $3$-manifold invariants, including those of Hennings \cite{hennings1996} and Kerler--Lyubashenko \cite{kerler2001}. See \cite{derenzi2018,derenzi2023}. Further examples of $3$-manifold invariants associated to non-degenerate relative pre-modular categories are studied in \cite{AGP, beliakova2021b, ha2018,bao2022}. For connections between standard and renormalized Reshetikhin--Turaev invariants of links and 3-manifolds, see \cite{costantino2015b,costantino2021,derenzi2020,mori2022}. 

The further extension of renormalized $3$-manifold invariants to three dimensional topological quantum field theories (TQFTs) was first accomplished in the case of weight modules over $\sunrolled$ \cite{BCGP}. The TQFTs have interesting features, including the extension of Reidemeister torsion to a TQFT and the possibility of producing representations of mapping class groups that are faithful modulo their centers. A general framework for the construction of (extended) TQFTs from \emph{relative modular categories} was given by De Renzi \cite{derenzi2022}, generalizing the so-called universal construction of semisimple TQFTs \cite{blanchet1995}. See \cite{blanchet2021} for an overview of this circle of ideas. Further examples of TQFTs from renormalized invariants are constructed and studied in \cite{derenzi2022b,geerYoung2022}.

At present, the main source of relative modular categories is the representation theory of unrolled quantum groups, thereby making this class of quantum groups central to non-semisimple topology. Motivated by the success of rational conformal field theoretic techniques in semisimple topology, a number of authors have pursued a conjectural logarithmic variant of the Kazhdan--Lusztig correspondence, which asserts an equivalence between categories of weight modules over unrolled quantum groups and modules over non-rational, or logarithmic, vertex operator algebras. Much progress has been made in the case of $\sunrolled$, where connections between weight modules and the singlet, triplet and Feigin--Tipunin algebras have been found \cite{creutzig2018,creutzig2022}. 

Finally, there has been exciting progress in connecting non-semisimple mathematical TQFTs to physical quantum field theories. This can be seen as a non-semisimple generalization of the celebrated connection between compact Chern--Simons theory and Reshetikhin--Turaev TQFTs \cite{witten1989,reshetikhin1991invariants}. The case of TQFTs arising from $\sunrolled$ is studied in \cite{creutzig2021}, where it is connected to a topological twist of $3$d $\mathcal{N}=4$ Chern--Simons matter theory with gauge group $SU(2)$. Similarly, TQFTs arising from an unrolled quantization of the Lie superalgebra $\mathfrak{gl}(1 \vert 1)$ were shown in \cite{geerYoung2022} to be related to supergroup Chern--Simons theories with gauge group $\mathfrak{psl}(1 \vert 1)$ and $U(1 \vert 1)$. A key feature in both physical realizations is the presence of global symmetry groups, allowing the quantum field theories to be coupled to background flat connections. Further physical studies of such quantum field theories can be found in \cite{gukov2021,jagadale2022}. In condensed matter physics, Levin and Wen used unitary spherical fusion categories to give a mathematical foundation of topological order and string-net condensation \cite{levin2005}. Recently, this construction was extended to the setting of the non-semisimple category of weight modules over $\sunrolled$ \cite{geer2022,geer2022b}.

\newcommand{\etalchar}[1]{$^{#1}$}

\end{document}